\documentclass[a4paper,12pt]{amsart}
\usepackage{amsfonts,amsthm,amsmath,amssymb,verbatim}

\usepackage{graphicx}
\usepackage[active]{srcltx}
\usepackage[all,cmtip]{xy}

\newtheorem{thm}{Theorem}[section]
\newtheorem{prop}[thm]{Proposition}
\newtheorem{lemma}[thm]{Lemma}
\newtheorem{keylemma}[thm]{Key Lemma}
\newtheorem{cor}[thm]{Corollary}

\theoremstyle{remark}

\newtheorem{example}[thm]{Example}
\newtheorem{remark}[thm]{Remark}

\makeatletter
\newenvironment{varproof}[1][\proofname]{\par
  \pushQED{\qed}%
  \normalfont \topsep6\p@\@plus6\p@\relax
  \trivlist
  \item[\hskip\labelsep
        \itshape
    #1]\ignorespaces
}{%
  \popQED\endtrivlist\@endpefalse
}
\makeatother

\def\C{\mathbb{C}}
\def\R{\mathbb{R}}
\def\Z{\mathbb{Z}}
\def\P{\mathbb{P}}

\def\Sym{{\rm Sym}}

\def\Volume{{\rm Volume}}

\def\d{\partial}

\def\Sf{\mathfrak{S}}

\def\Sc{\mathcal{S}}
\def\Lc{\mathcal{L}}

\def\a{\alpha}
\def\b{\beta}
\def\g{\gamma}
\def\l{\lambda}
\def\G{\Gamma}
\def\eps{\varepsilon}

\def\emptyset{\varnothing}

\title{Schubert calculus and Gelfand--Zetlin polytopes}

\author{Valentina Kiritchenko}
\email{vkiritchenko@yahoo.ca}
\thanks{The authors were supported by RFBR grant 10-01-00540-a, AG Laboratory NRU-HSE,
MESRF grant, ag. 11.G34.31.0023, RF Innovation Agency grant 02.740.11.0608,
Simons Foundation (ES,VT), Dynasty Foundation (VK), Deligne fellowship (VT),
MESRF grants MK-2790.2011.1 (VT), 16.740.11.0307 (VK,ES), RFBR grants 11-01-00289-a (ES),
11-01-00654-a (VT), RFBR-CNRS grants 10-01-93110-a (VK), 10-01-93111-a (ES)}

\address{Laboratory of Algebraic Geometry and Faculty of Mathematics\\ Higher School of Economics\\
Vavilova St. 7, 112312 Moscow, Russia}
\address{Institute for Information Transmission Problems, Moscow, Russia}

\author{Evgeny Smirnov}
\email{evgeny.smirnov@gmail.com}

\address{Laboratory of Algebraic Geometry and
Faculty of Mathematics\\ Higher School of Economics\\
Vavilova St. 7, 112312 Moscow, Russia}
\address{Laboratoire J.-V.~Poncelet (UMI 2615 du CNRS) and Independent University of Moscow, Moscow, Russia}

\author{Vladlen Timorin}
\email{vtimorin@hse.ru}

\address{Laboratory of Algebraic Geometry and Faculty of Mathematics\\ Higher School of Economics\\
Vavilova St. 7, 112312 Moscow, Russia}
\address{Independent University of Moscow, Moscow, Russia}


\date{}

\begin{document}

\begin{abstract}
  We describe a new approach to the Schubert calculus on complete flag varieties
using the volume polynomial associated with Gelfand--Zetlin polytopes.
This approach allows us to compute the intersection products of Schubert cycles
by intersecting faces of a polytope.
\end{abstract}

\maketitle

\section{Introduction}

In this paper, we explore the connection between the Schubert calculus
and the volume polynomial on spaces of convex polytopes.
We give various representations of Schubert cycles in a complete flag
variety by sums of faces of the Gelfand--Zetlin polytope.
Our work is motivated by the rich interplay between algebraic geometry and
convex polytopes, originally explored for toric varieties and recently
extended to a more general setting in \cite{KK}.

One of our main tools is a construction of \cite{PKh}, which, to every convex polytope
$P\subset\R^d$, associates a graded commutative ring $R_P$ (called the {\em polytope ring})
satisfying the Poincare duality (see \cite{T} or Section \ref{s.polytopes}).
For an {\em integrally simple} polytope $P$ (simple means that there
are exactly $d=\dim(P)$ edges meeting at each vertex, and integrally simple means
that primitive integer vectors parallel to the edges generate the lattice $\Z^d$),
the ring $R_P$ is isomorphic to the Chow ring of the corresponding smooth toric variety $X_P$ \cite{PKh}.
Faces of $P$ give rise to certain elements of $R_P$, which generate $R_P$ as an additive group.
If $[F]$ is the element of $R_P$ corresponding to a face $F$, then
$[F]\cdot [G]=[F\cap G]$ in $R_P$, provided that $F$ and $G$ are transverse.
Single faces of $P$ represent cycles given by the closures of the torus orbits in $X_P$.
In this paper, we are primarily interested in the case, where $P$ is not simple.
Kiumars Kaveh has related the polytope rings of some non-simple
polytopes to the Chow rings of smooth non-toric spherical varieties \cite{Kaveh}.
In particular, he observed that the ring $R_P$ for the {\em Gelfand--Zetlin polytope}
$P=P_\lambda$ (which is not simple) associated with a strictly dominant weight
$\l=(\l_1,\ldots,\l_n)\in\Z^n$ of the group $GL_n(\C)$ is isomorphic to the Chow ring
of the variety $X$ of complete flags in $\C^n$.

When $P$ is not simple, there is no straightforward correspondence between faces of $P$ and elements of $R_P$.
One of the results of the present paper is a general construction that
with every element of $R_P$ associates a linear combination of faces of $P$
(though not every face of $P$ corresponds to an element of $R_P$).
Namely, we embed the ring $R_P$ into a certain $\Z$-module $M_P$, whose
elements can be regarded as linear combinations of arbitrary faces of $P$ modulo some relations
(see Section \ref{s.polytopes}).
The module $M_P$ depends on the choice of a {\em resolution} of $P$.
On the algebro-geometric level, $R_P$ can be regarded as the subring of the Chow ring of the singular toric variety $X_P$
generated by the Picard group and $M_P$ can be constructed using a resolution of
singularities for $X_P$.
However, we describe $M_P$  in elementary terms using convex geometry.
A crucial feature of such representations by sums of faces is that we can still multiply
elements of $R_P$ by intersecting faces (assuming the faces we intersect are transverse).

While our construction applies to any convex polytope $P$, it is especially interesting to
study the case, where $P=P_\lambda$ is a Gelfand--Zetlin polytope,
due to the isomorphism $R_P\simeq CH^*(X)$ for the flag variety $X$.
Recall that $CH^*(X)$ (as a group) is a free abelian group with the basis of Schubert cycles.
In particular, our construction allows to represent Schubert cycles as linear combinations of faces of the Gelfand--Zetlin polytope in many different ways (see Theorem \ref{t.FK}, Proposition \ref{P:linrels}, Corollary \ref{c.FK}), which has applications to Schubert calculus.

The relation between Schubert varieties and faces of the Gelfand--Zetlin polytope
was first investigated in \cite{KThesis}, and then by different methods also in
\cite{KM} and \cite{VK}.
The approach of \cite{KM} (via degenerations of Schubert varieties to subvarieties of a singular toric variety) seems to be the closest to ours.
However, the results of \cite{KM} can not be directly used for purposes of Schubert calculus,
since only one degeneration is constructed for each Schubert variety.
The polytope ring $R_P$ and the $\Z$-module $M_P$ allow us to bypass toric degenerations and do all calculations with faces 
directly in $R_P$.

Given two Schubert cycles $[X^w]$ and $[X^{w'}]$, we can represent
$[X^w]$ and $[X^{w'}]$ as sums of faces so that every face appearing in the decomposition
of $[X^w]$ is transverse to every face appearing in the decomposition of $[X^{w'}]$ (see Corollary \ref{c.Richardson}).
This allows us to represent the intersection of any two Schubert cycles by
linear combinations of faces with nonnegative coefficients.
This might lead to a positive formula for the structure constants (which are triple products $[X^w][X^{w'}][X^{w''}]$) by counting vertices of the Gelfand--Zetlin polytope.

The connection between Schubert calculus and Gelfand--Zetlin polytopes stems from the representation theory of $GL_n(\C)$.
Recall that by definition of Gelfand--Zetlin polytopes the integer points inside and at the boundary of $P_\l$ parameterize a natural basis (a {\em Gelfand--Zetlin basis}) in the irreducible highest weight $GL_n$--module $V_\l$ with the highest weight $\l$.
In particular, with every integer point $z\in P_\l$ we can associate its weight $p(z)$ in the character lattice of $GL_n$.
To derive some presentations of Schubert cycles by sums of faces we establish the following relation between {\em Demazure submodules} of $V_\l$ and faces of $P_\l$.
For every Schubert variety $X^w$ and a strictly dominant weight $\l$, we realize the corresponding Demazure character as the exponential sum $\sum e^{p(z)}$ where $z$ runs over integer points  in the union of all {\em rc-faces, or reduced Kogan faces} (see Section \ref{s.GZ}) of $P_\lambda$ with permutation $w$ (see Theorem \ref{t.Demazure}).
This generalizes the identity from \cite[Corollary 15.2]{PS} for the Demazure character
of a {\em $132$--avoiding}, or {\em Kempf}, permutation $w$
(such permutations are also called \emph{dominant}, but we will use the term ``Kempf'' instead).
Note that a permutation is Kempf if and only if there is a unique face
with this permutation (see \cite[Proposition 2.3.2]{KThesis}), and this is exactly
the face considered in \cite{PS}.

To prove our formula for the Demazure character we use elementary convex geometry together with a simple
combinatorial procedure for dealing with divided difference operators (called {\em mitosis})
introduced in \cite{KnM}  (see also \cite{Mi} for an elementary exposition).
In particular, our proof yields a geometric realization of mitosis (see Subsection \ref{ss:paramitosis}).
As a byproduct, we construct a minimal realization of a simplex as a cubic complex different
from those previously known (see Proposition \ref{p.cubic}).

This paper is organized as follows.
In Section~\ref{s.polytopes}, we recall the definition of the polytope ring $R_P$, discuss
its properties and construct a module $M_P$ for a non-simple $P$.
In Section~\ref{s.GZ}, we study the polytope rings of the Gelfand--Zetlin polytopes.
In Section~\ref{s.cycles}, we represent Schubert cycles by faces.
In Section~\ref{s.Demazure}, we give formulas for Demazure characters, Hilbert functions
and degrees of Schubert varieties in terms of faces and deduce from these formulas
some of the results of Section~\ref{s.cycles}.
In Section~\ref{s.keylemma}, we introduce a simple geometric version of mitosis
({\em paramitosis}) and use it to prove formulas of
Section~\ref{s.Demazure} for Demazure characters.

{\em Acknowledgements.} This project was started when the first, second and third
authors, respectively, were affiliated with the Max Planck Institute for Mathematics
(MPIM), Hausdorff Center for Mathematics in Bonn and Jacobs University Bremen.
The project was continued when the first and the third author visited the
Freie Universit\"at Berlin and the MPIM, Bonn. We would like to thank these institutions
for hospitality, financial support and excellent working conditions.

The authors are grateful to Michel Brion, Askold Khovanskii and Allen Knutson for useful
discussions.

\section{Polytope ring} \label{s.polytopes}

\subsection{Rings associated with polynomials}
Following \cite{PKh}, we associate a graded commutative ring with any polynomial.
We will later specialize to the case of the volume polynomial on a
space of polytopes with a given normal fan.
Let $\Lambda_f$ be a lattice, i.e. a free $\Z$-module, and $f$ a homogeneous polynomial
on the real vector space $V_f=\Lambda_f\otimes\R$ containing the lattice $\Lambda_f$.
The symmetric algebra $\Sym(\Lambda_f)$ of $\Lambda_f$ can be
thought of as the ring of differential operators with constant
integer coefficients acting on $\R[V_f]$, the space of all polynomials on $V_f$.
If $D\in\Sym(\Lambda_f)$ and $\phi\in\R[V_f]$, then we write
$D\phi\in\R[V_f]$ for the result of this action.
Define $A_f$ as the homogeneous ideal in $\Sym(\Lambda_f)$ consisting of all
differential operators $D$ such that $Df=0$.
Set $R_f=\Sym(\Lambda_f)/A_f$.
We call this ring {\em the ring associated with the polynomial $f$}.

Let $\Lambda_g$ be another lattice, and $\sigma:\Lambda_g\to\Lambda_f$
a homomorphism of lattices.
Define the polynomial $g\in\R[V_g]$ as $\sigma^*(f)=f\circ\sigma$.
We want to describe a relation between the associated rings $R_f$ and $R_g$.
Unfortunately, there is no natural homomorphism between these rings.
However,

\begin{prop}
\label{P:Mod}
There is a natural abelian group $M_{f,g}$, a natural epimorphism
$\pi:R_f\to M_{f,g}$ and a natural monomorphism
$\iota:R_g\to M_{f,g}$ such that
$$
\pi(\tilde\alpha\tilde\beta)=\iota(\alpha\beta)
$$
whenever $\pi(\tilde\alpha)=\iota(\alpha)$ and $\pi(\tilde\beta)=\iota(\beta)$.
\end{prop}

This proposition can be used in the following way.
Elements of $R_g$ can be embedded naturally into $M_{f,g}$.
Although elements of $M_{f,g}$ cannot be multiplied in general,
we consider the lifts to $R_f$ of two elements coming from $R_g$,
multiply them in $R_f$, and project the product back to $M_{f,g}$.
In many cases, this is easier than multiplying two elements of $R_g$
directly.

\begin{proof}
  Consider a $\Z$-submodule $A_{f,g}$ of $\Sym(\Lambda_f)$ consisting
  of all operators $D$ such that $\sigma^*(Df)=0$.
  Set $M_{f,g}=\Sym(\Lambda_f)/A_{f,g}$.
  Clearly, $A_f\subset A_{f,g}$; thus we obtain a natural projection
  $\pi:R_f\to M_{f,g}$.
  Let $\sigma_*:\Sym(\Lambda_g)\to\Sym(\Lambda_f)$ be the homomorphism
  induced by the map $\sigma$.
  For a differential operator $D\in\Sym(\Lambda_g)$, let $[D]$
  denote the class of $D$ in the ring $R_g$.
  We define $\iota([D])$
  as the class in $M_{f,g}$ of the operator $\sigma_*(D)$.

  To verify that $\iota([D])$ is well defined, we need the following formula
  $$
  \sigma^*(\sigma_*(D)\phi)=D\sigma^*\phi
  $$
  for every $\phi\in\R[V_f]$.
  Indeed, this formula is obviously true if $D\in\Lambda_g$, and
  both parts of this formula depend multiplicatively on $D$.
  In particular, we have
  $$
  \sigma^*(\sigma_*(D)f)=D\sigma^*f=Dg,
  $$
  which is equal to zero whenever $D$ is in $A_g$.
  It follows that the element $\iota([D])$ is well defined: if $D\in A_g$, then
  $\sigma_*(D)\in A_{f,g}$.
  It also follows from the same formula that $\iota$ is injective:
  if $\iota([D])=0$, i.e. $\sigma_*(D)\in A_{f,g}$, then $D\in A_g$.

  It remains to prove that $\pi(\tilde\alpha\tilde\beta)=\iota(\alpha\beta)$
  whenever $\pi(\tilde\alpha)=\iota(\alpha)$ and $\pi(\tilde\beta)=\iota(\beta)$.
  But this is an immediate consequence of the formula
  $\sigma_*(DE)=\sigma_*(D)\sigma_*(E)$.
\end{proof}

\subsection{The volume polynomial}
Consider the set of all convex polytopes of dimension $d$ in $\R^d$.
This set can be endowed with the structure of a commutative semigroup using {\em Minkowski sum}
$$
P_1+P_2=\{x_1+x_2\in\R^n\ |\ x_1\in P_1,\ x_2\in P_2\}
$$
It is not hard to check that this semigroup has cancelation property.
We can also multiply polytopes by positive real numbers using dilation:
$$
\lambda P=\{\lambda x\ |\ x\in P\},\quad \lambda\ge 0.
$$
Hence, we can embed the semigroup of convex polytopes into its Grothendieck group $V$,
which is a real (infinite-dimensional) vector space.
The elements of $V$ are called {\em virtual polytopes}.
Recall that two convex polytopes are called {\em analogous} if they have the same normal fan,
i.e. there is a one-to-one correspondence between the faces of $P$ and the
faces of $Q$ such that any linear functional, whose restriction to $P$
attains its maximal value at a given face $F\subseteq P$ has the
property that its restriction to $Q$ attains its maximal value at the corresponding face of $Q$
(the set of linear functionals, whose restrictions to $P$ attain their maximal values at a face
$F\subset Q$, form a cone $C_F$; the {\em normal fan} of $P$ is defined as the
set of cones $C_F$ corresponding to all faces $F\subseteq Q$).
A virtual polytope is said to be analogous to $P$ if it can be
represented as a difference of two convex polytopes analogous to $P$.
All virtual polytopes analogous to $P$ form a finite dimensional subspace $V_P\subset V$.
On the vector space $V$, there is a homogeneous polynomial $vol$ of degree $n$, called the
{\em volume polynomial}.
Fix a constant (translation invariant) volume form on $\R^d$.
If an integer lattice $\Z^d\subset\R^d$ is fixed, we will always choose
this volume form to take value 1 on the fundamental parallelepiped of $\Z^d$.
The volume form on $\R^d$ being fixed, the volume
polynomial on the space $V$ is uniquely characterized by the property that its
value $vol(P)$ on any convex polytope $P$ is equal to the volume of $P$.
We will be interested in the restriction $vol_P$ of the volume polynomial
$vol$ to the subspace $V_P$ of all virtual polytopes analogous to $P$.

Consider an integer convex polytope $P$ (i.e. a convex polytope with integer vertices)
of dimension $n$, not necessarily simple.
Let $\Lambda_P$ be a lattice in $V_P$ generated by some integer polytopes
analogous to $P$ (we do not assume that $\Lambda_P$ contains all
integer polytopes analogous to $P$; thus this lattice may depend on some extra
choices rather than only on $P$).
Suppose that $Q$ is a convex polytope with integer vertices, whose
normal fan is a simplicial subdivision of the normal fan of $P$.
In this case, $Q$ is called a {\em resolution} of $P$ (note that,
since the normal fan of $Q$ is simplicial, the polytope $Q$ is simple).
With the volume polynomial $vol_P$ restricted to the lattice $\Lambda_P$,
we associate the {\em polytope ring} $R_P:=R_{vol_P}$.
Similarly, for the simple polytope $Q$, we consider the ring $R_Q:=R_{vol_Q}$ associated
with the volume polynomial $vol_Q$ on the lattice $\Lambda_Q$
(we always assume that this lattice is generated by {\em all} integer polytopes
analogous to $Q$).
We will use the $\Z$-module $M_{Q,P}:=M_{vol_Q,vol_P}$ introduced in Proposition \ref{P:Mod}
together with the homomorphisms $\iota:R_P\to M_{Q,P}$ and $\pi:R_Q\to M_{Q,P}$.
Since $\iota$ is a canonical embedding, we will identify
elements of $R_P$ with their $\iota$-images in $M_{Q,P}$.
With every face $\tilde F$ of $Q$, we can associate a face $F$ of $P$ called the
{\em $P$-degeneration} of $\tilde F$ (or just {\em degeneration} if $P$ is fixed).
A face $F$ of $P$ is called {\em regular (with respect to $Q$)} if
there is only one face $\tilde F$ of $Q$ such that $F$ is the degeneration of
$\tilde F$.

\begin{prop}
\label{P:reg}
  Suppose that $v$ is a simple vertex of $P$, i.e. exactly $d=\dim(P)$
  facets of $P$ meet at $v$.
  Moreover, suppose that no facet of $Q$ degenerates to a face of smaller dimension.
  Then any face of $P$ containing $v$ is regular.
\end{prop}

\begin{proof}
 Let $\G_1$, $\dots$, $\G_d$ be all facets of $P$ containing the vertex $v$ (which are clearly regular).
 Denote by $\tilde\G_1$, $\dots$, $\tilde\G_d$ the corresponding
 (parallel) facets of $Q$.
 Note that the intersections of different subsets of $\{\G_1,\dots,\G_d\}$
 are different faces of $P$.
 Clearly, any intersection of facets $\tilde\G_i$ degenerates into
 the intersection of the corresponding facets $\G_i$ (which has
 the same dimension), and all other faces of $Q$ degenerate to
 faces of $P$ not containing the vertex $v$.
 This proves the desired statement.
\end{proof}

\subsection{Structure of polytope rings}
We now give more details on the structure of the ring $R_Q$.
For every facet $\Gamma$ of $Q$, there is a differential operator
$\d_\Gamma\in\Sym(\Lambda_Q)$ such that, for every convex polytope $Q'$
analogous to $Q$, the number $\d_\Gamma vol_Q(Q')$ is
the $(n-1)$-dimensional volume of the facet of $Q'$ parallel to $\Gamma$.
The ideal $A_Q:=A_{vol_Q}$ is very easy to describe.
It is generated (as an ideal) by the following two groups of differential
operators \cite{T}:
\begin{itemize}
  \item the images of integer vectors $a\in\Z^d$ under the natural
  inclusion of $\Z^d$ into $\Lambda_Q=\Sym^1(\Lambda_Q)$ such that $Q+a$ is the
  parallel translation of $Q$ by the vector $a$;
  \item the operators of the form $\d_{\Gamma_1}\dots\d_{\Gamma_k}$,
  where $\Gamma_1\cap\dots\cap\Gamma_k=\varnothing$.
\end{itemize}

The volume polynomial on the spaces $V_Q$ was previously used in \cite{PKh}
to describe the cohomology rings of smooth toric varieties.
We briefly recall this description.
Every integer polytope $Q$ defines a polarized toric variety $X_Q$.
If $Q$ is integrally simple, then $X_Q$ is smooth.
In this case, the Chow ring of $X_Q$ (or, equivalently, the cohomology ring
$H^{2*}(X_Q,\Z)$) is isomorphic to $R_Q$ \cite[1.4]{PKh}.

This description is very useful.
Firstly, it is functorial.
Secondly, it is clear from the definition that the nonzero homogeneous components of
the ring $R_Q$ have degrees $\le d$ (since the volume polynomial has degree $d$)
and that $R_Q$ has a non-degenerate pairing (Poincar\'e duality) defined by $(D_1,D_2):=D_1D_2(vol_Q)\in\Z$
for any two homogeneous elements $D_1$, $D_2\in\Sym(\Lambda_Q)$ of complementary degrees.
The Poincar\'e duality on the ring $R_Q$ is a key ingredient
in the proof of the isomorphism between $R_Q$ and $H^{2*}(X_Q,\Z)$ (see \cite{Kaveh} for more details).
Note that there is another functorial description \cite{Br97} of the Chow ring of $X_Q$
via piecewise polynomial functions on fans but for this description the upper bound on
the degrees and the Poincar\'e duality are harder to check directly.
Also, a first known (non-functorial) description of the Chow ring (by generators and
relations) follows easily from the definition of the ring $R_Q$ (see e.g. \cite{T}).
So it seems that the polytope rings give the most convenient description
of the Chow rings of smooth toric varieties.

Note that if a polytope $P$ is not simple, the ring $R_P$ makes sense, has all
nonzero homogeneous components in degrees $\le d$ and satisfies the Poincar\'e duality.
However, its relation to the Chow ring of (now singular) toric
variety $X_P$ is unclear partly because the latter no longer enjoys the Poincar\'e duality.
On the other hand, the ring $R_P$ for non-simple polytopes is sometimes related to the Chow rings of
smooth non-toric varieties as was noticed by  Kaveh \cite{Kaveh}.

We now discuss some important properties of the isomorphism $R_Q\simeq CH^*(X_Q)$
for a simple polytope $Q$.
This isomorphism allows us to identify the algebraic cycles on $X_Q$
with the linear combinations of the faces of $Q$.
The dimension of the space $V_Q$ is equal to the number $N(Q)$ of facets of $Q$
(since we can shift all support hyperplanes of $Q$ independently).
Note that for a non-simple polytope $P$ the dimension of $V_P$ is strictly less than $N(P)$
(e.g. if $P$ is an octahedron, then $V_P$ has dimension 4).
For simple $Q$, the space $V_Q$ has natural coordinates called the {\em support numbers}.
There are as many support numbers as facets of $Q$.
The support numbers are defined by fixing $N=N(Q)$ linear functionals $\xi_\G$
on $\R^d$ corresponding to facets $\G$ of $Q$ such that every facet $\G$ of $Q$
is contained in the hyperplane $\xi_\G(x)=H_\G$ for some constants $H_\G$,
and the polytope $Q$ satisfies the inequalities $\xi_\G(x)\le H_\G$.
If $\G_1$, $\dots$, $\G_N$ are all facets of $Q$, then any collection of real numbers
$(H_{\G_1},\ldots,H_{\G_N})$ defines a unique  (possibly virtual) polytope in $V_Q$.
When dealing with integer polytopes, we always choose $\xi_\G$ to be a primitive
integer covector orthogonal to $\G$.
In this case, $H_\G$ is (up to a sign) the integer distance between the
origin and the hyperplane containing $\G$.

If we choose the volume forms and the linear functionals $\xi_\G$ consistent
with the integer lattice (in the sense explained above),
then the differential operators $\d_\G$ coincide with
the partial derivatives with respect to the support numbers $H_\G$.
For a face $F=\Gamma_1\cap\dots\cap\Gamma_k$ of codimension $k$,
we set $\d_F=\d_{\Gamma_1}\dots\d_{\G_k}$, and denote by $[F]$ the class of $\d_F$ in the ring $R_Q$.
The elements $[F]$ corresponding to the faces of $Q$ generate $R_Q$ as an Abelian group.
Moreover, it suffices to take some special faces, called {\em separatrices} in \cite{T}.
There is an explicit algorithm to represent the product $[F_1]\cdot [F_2]\in R_Q$
as a linear combination of faces, i.e. of elements of the form $[F]$ corresponding
to faces $F$ of $Q$.
This algorithm resembles the well-known algorithm from intersection theory:
we need to replace $[F_1]$ by a linear combination of faces that are transverse to $F_2$.
The linear relations between facets of $Q$ follow immediately from the description
of the ideal $A_Q$ given above.
They have the form
$$
\sum\xi_\G(a)[\G]=0, \eqno(4)
$$
where $a\in\R^d$ is any vector, and the sum is over all facets of $Q$.
Indeed, the volume polynomial is invariant under parallel translations.
Therefore, the $t$-derivative of $vol(\cdot+ta)$ is zero (where $\cdot$ replaces
any fixed element of $V_Q$).
By the chain rule, this derivative is equal to $\sum\xi_\G(a)\d_\G vol(\cdot)$.
Any linear relation between the elements $[\G]$ has this form (see \cite{T}).

If $Q$ is a resolution of $P$, we will be interested in representations of
elements $\alpha\in R_P$ by linear combinations of faces of $Q$ i.e.
in the following form
$$
\alpha=\pi\left(\sum [F]\right),
$$
where the summation is over some set of faces of $Q$. Then Proposition \ref{P:Mod}
allows us to compute the product of two elements $\alpha, \alpha'\in R_P$ as follows.
If we find a representation
$$
\alpha'=\pi\left(\sum [F']\right),
$$
such that all $F'$ are transverse to all $F$, then
$$
\alpha\cdot\alpha'=\pi\left(\sum [F\cap F']\right).
$$

In the sequel, we will also use the following lemma, which is a direct corollary
of the definition of $M_{Q,P}$.
\begin{lemma} \label{L:Poincare}Let $\a$ and $\b$ be two homogeneous elements in $R_Q$ of the same degree.
We have $\pi(\a)=\pi(\b)$ in $M_{Q,P}$ iff
$\pi(\a \g)=\pi(\b\g)$ for all homogeneous $\g\in R_Q$ of complementary degree such that
$\pi(\g)\in R_P\subset M_{Q,P}$.
\end{lemma}

\subsection{Example: Gelfand--Zetlin polytopes in $\R^3$.}
Consider the polytope $P$ in $\R^3$ given by the following linear inequalities:
$$
a\le x\le b,\quad b\le y\le c,\quad x\le z\le y.
$$
This is a 3-dimensional Gelfand--Zetlin polytope, see Figure 1.
The defining system of linear inequalities for $P$ is usually represented
schematically as follows:
$$
\begin{array}{ccccc}
  a& &b& &c\\
   &x& &y&\\
   & &z& &
\end{array}
$$

\begin{figure}
\includegraphics[width=5cm]{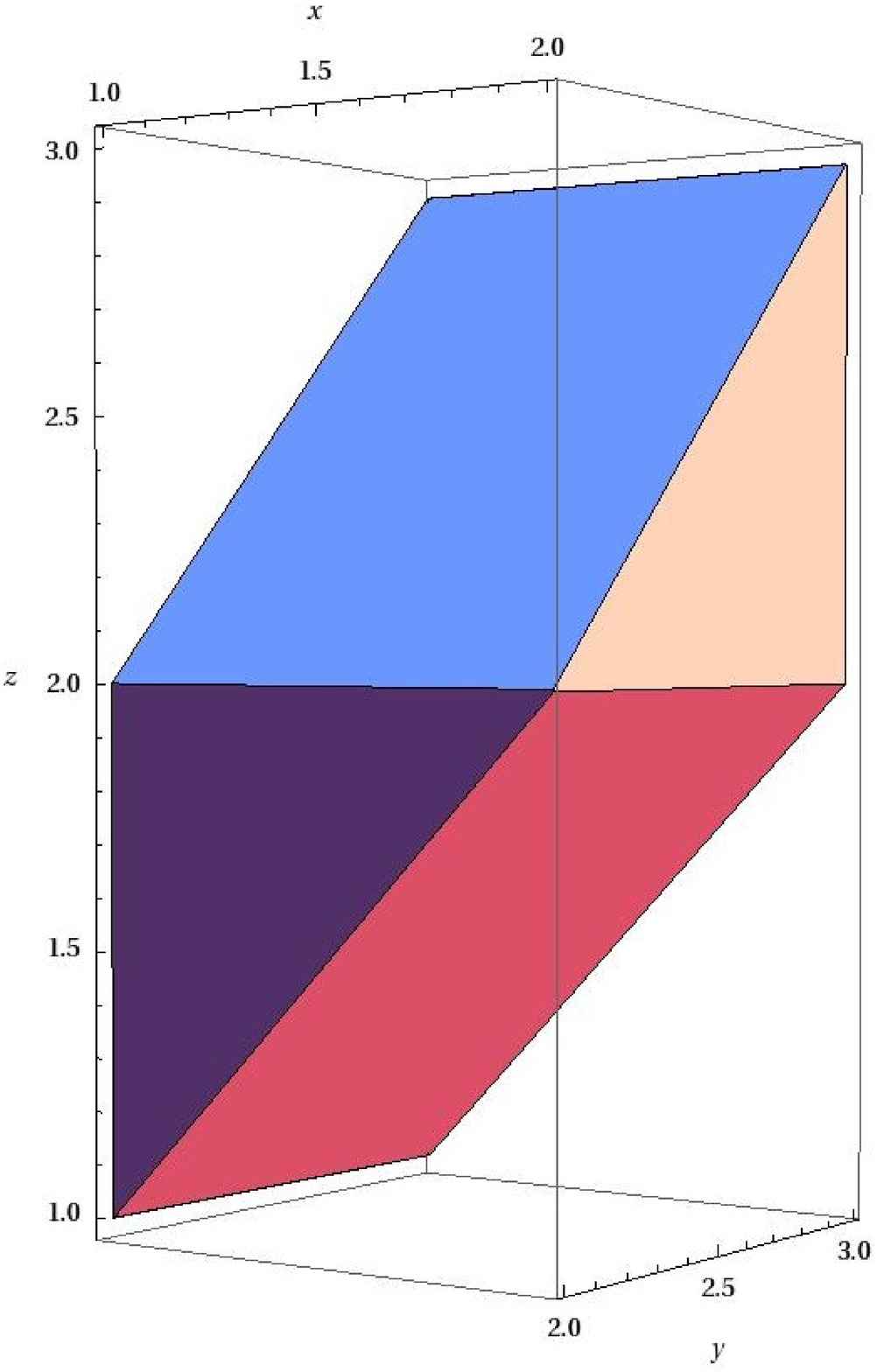}\qquad
\includegraphics[width=5cm]{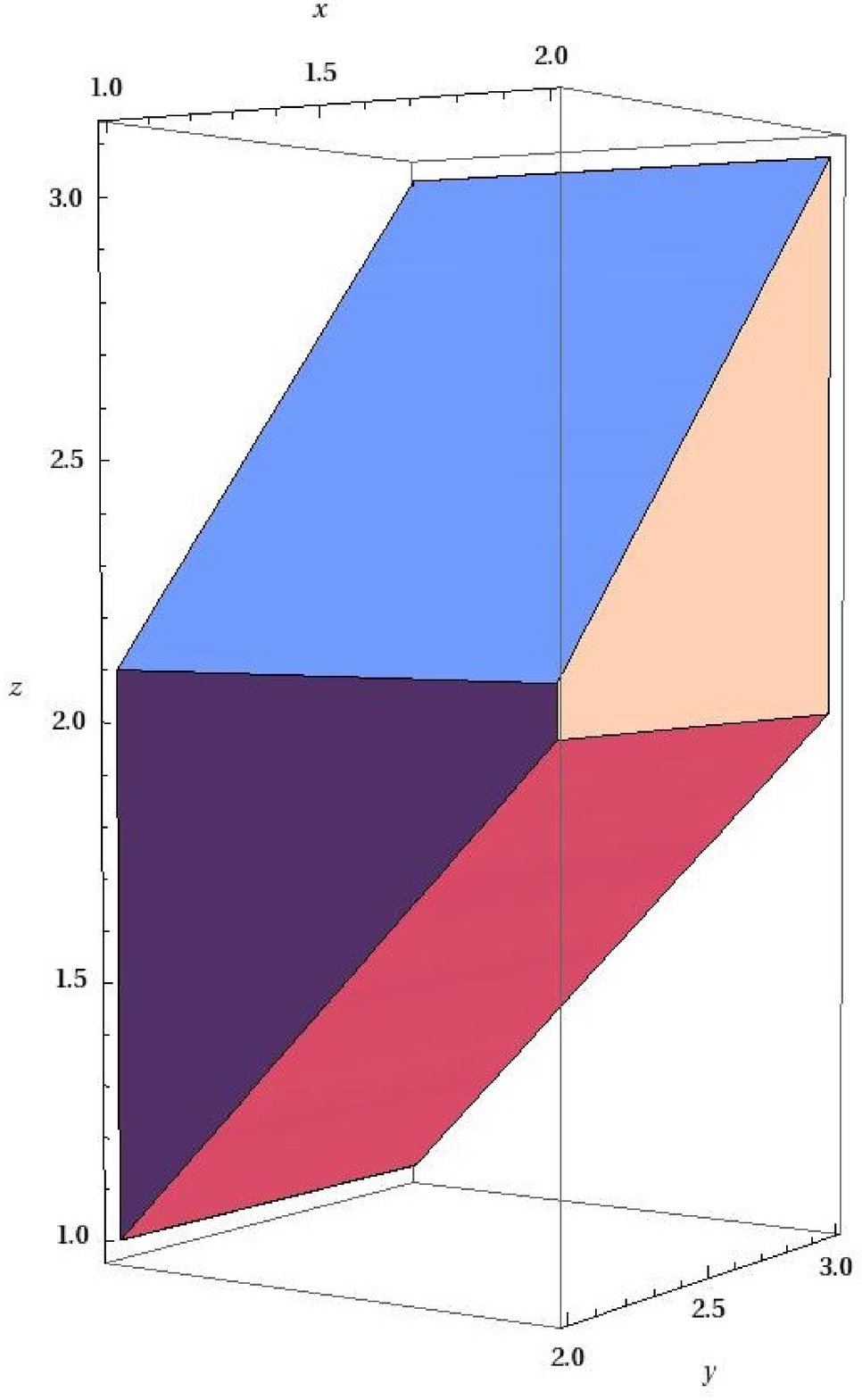}
\caption{A Gelfand--Zetlin polytope for $GL_3$ and its resolution}
\end{figure}

The polytope $P$ can be obtained from the parallelepiped $[a,b]\times [b,c]\times [a,c]$
by removing two prisms:
$$
\{a\le z<x\le b,\ b\le y\le c\},\quad \{b\le y<z\le c,\ a\le x\le b\}.
$$
Therefore, the volume of $P$ is equal to
$$
(b-a)(c-b)(c-a)-\frac{(b-a)^2(c-b)}2-\frac{(c-b)^2(b-a)}2=\frac 12(b-a)(c-b)(c-a)
$$
(it can also be seen geometrically, without any computation, that the part
we are removing is half the volume of the entire parallelepiped).
The ring $R_P$ is spanned by the classes of partial differentiations
$\d_a$, $\d_b$ and $\d_c$.
Moreover, since the volume of $P$ will not change if we shift $a$, $b$ and $c$
simultaneously by the same real number, we have $\d_a+\d_b+\d_c=0$ in $R_P$.
A distinguished set of additive generators of $R_P$ is given by Schubert
polynomials in $-\d_a$ and $-\d_b$, i.e.

$$
\Sf_{s_1s_2s_1}=-\d_a^2\d_b,\quad
\Sf_{s_1s_2}=\d_a\d_b,\quad
\Sf_{s_2s_1}=\d_a^2,\quad
\Sf_{s_2}=-\d_a-\d_b,
$$
$$
\Sf_{s_1}=-\d_a,\quad
\Sf_{id}=1.
$$

Now consider a simple polytope $Q$ given by the following inequalities:
$$
a\leq x\leq b,\quad b\leq y\leq c,\quad x\leq z\leq  y+\eps,
$$
where $\eps>0$ is a fixed small number.
The polytope $Q$ can also be obtained from the parallelepiped
$[a,b]\times [b,c]\times [a,c+\eps]$
by removing two prisms:
$$
\{a\le z<x\le b,\ b\le y\le c\},\quad \{b\le y<z-\eps\le c,\ a\le x\le b\}.
$$
Therefore, the volume of $Q$ is equal to
$$
(b-a)(c-b)(c-a+\eps)-\frac{(b-a)^2(c-b)}2-\frac{(c-b)^2(b-a)}2=
$$
$$
=\frac 12(b-a)(c-b)(c-a)+\eps(b-a)(c-b).
$$
This is a polynomial in $a$, $b$, $c$ and $\eps$.

The ring $R_Q$ is multiplicatively generated by partial differentiations
$\tilde\d_a$, $\tilde\d_b$ and $\tilde\d_c$ (the tildes are just to
distinguish these elements of $R_Q$ from elements $\d_a$, $\d_b$, $\d_c\in R_P$).
We have
$$
\tilde\d_a=-[x=a],\quad
\tilde\d_b=-[y=b]+[x=b],\quad
\tilde\d_c=[y=c].
$$
Formula (4) gives three linear relations between facets of $Q$:
$$
-[x=a]+[x=b]+[z=x]=0,
$$
$$
-[y=b]-[z=y+\eps]+[y=c]=0,
$$
$$
-[x=z]+[z=y+\eps]=0.
$$
We can represent the Schubert polynomials of $\d_a$ and $\d_b$ as the $\pi$-images
of certain elements of $R_Q$ as follows:
$$
\Sf_{s_1}=\pi[x=a],\quad \Sf_{s_2}=\pi([x=a]+[y=b]),
$$
$$
\Sf_{s_2s_1}=\pi[x=z=a],\quad \Sf_{s_1s_2}=\pi[x=a,y=b].
$$
All faces of $Q$ that appear in the right hand sides of these identities
degenerate to regular faces of $P$.
For instance, the expression for $\Sf_{s_2s_1}$ is obtained as follows:
$$
\Sf_{s_2s_1}=\d_a^2=\pi(\tilde\d_a^2)=\pi([x=a]\cdot [x=a])=
$$
$$
=\pi([x=a]\cdot([x=b]+[z=x]))
=\pi([x=a]\cdot[x=b]+[x=z=a]).
$$
The first term in the right hand side vanishes, because the faces $\{x=a\}$ and
$\{x=b\}$ are disjoint.

In this way, it is easy to justify all heuristic calculations with faces in \cite[Section 4]{VK}.

\section{Gelfand--Zetlin polytope and its ring}
\label{s.GZ}

\subsection{Gelfand--Zetlin polytope}
We now consider the ring $R_P$ for the {\em Gelfand--Zetlin polytope}
$P=P_\lambda$  associated with a strictly dominant weight
$\l=(\l_1,\ldots,\l_n)\in\Z^n$ of the group $GL_n(\C)$, i.e.
with an $n$-tuple of integers $\l_i$ such that $\l_i<\l_{i+1}$ for all
$i=1,\dots,n-1$.
Recall that the Gelfand--Zetlin polytope
$P_\l$ is a convex integer polytope in $\R^{d}$,
where $d=n(n-1)/2$, with the property that the integer points inside
and at the boundary of $P_\l$ parameterize a natural basis in the
irreducible representation of $GL_n(\C)$ with the highest weight $\l$.
It can be defined by inequalities
$$
\begin{array}{ccccccccc}
\l_1&       & \l_2    &         &\l_3     &         &\ldots   &         &\l_n   \\
    &\l_{1,1}&         &\l_{1,2}  &         & \ldots  &         &\l_{1,n-1}&       \\
    &       &\l_{2,1} &         &\ldots   &         &\l_{2,n-2}&         &       \\
    &       &         & \ddots  &\ldots   &         &         &         &       \\
    &       &         &\l_{n-2,1}&         &\l_{n-2,2}&         &         &       \\
    &       &         &         &\l_{n-1,1}&         &         &         &       \\
\end{array}
\eqno{(GZ)}
$$
where
$(\l_{1,1},\ldots,\l_{1,n-1};\l_{2,1},\ldots,\l_{2,n-2};\ldots;
\l_{n-2,1},\l_{n-2,2};\l_{n-1,1})$ are coordinates in $\R^d$, and
 the notation
 $$
 \begin{array}{ccc}
  a &  &b \\
   & c &
 \end{array}
 $$
means $a\le c\le b$.
See Figure 1 for a picture of the Gelfand--Zetlin polytope for $G=GL_3$.
Note that Gelfand--Zetlin polytopes $P_\lambda$ and $P_\mu$ are analogous for
any two strictly dominant weights $\lambda$ and $\mu$.
For what follows, we set $P=P_\l$ for some strictly dominant weight $\l$,
and define $\Lambda_{P}$ as the lattice spanned by all Gelfand--Zetlin polytopes
$P_\mu$, where $\mu$ runs through all strictly dominant weights.
The correspondence $\mu\mapsto P_\mu$ establishes a natural isomorphism
between the lattices $\Z^n$ and $\Lambda_P$.
In other words, virtual polytopes in $\Lambda_P$ are parameterized by
arbitrary $n$-tuples of integers, not necessarily strictly increasing.
One can show that the ring $R_P$ does not change if $\Lambda_P$ is replaced by the lattice generated by {\em all}
polytopes analogous to $P_\l$ but we will not need this.

Recall that, to every complete flag $W=W_1\subset\dots\subset W^{n-1}$
in $\C^n$, one associates one-dimensional vector spaces $L_i(W)=W_i/W_{i-1}$.
The disjoint union of $\{W\}\times L_i(W)$ with its natural projection
to $X$ given by the formula $\{W\}\times L_i(W)\mapsto W$ has a structure
of a line bundle over $X$.
This line bundle $\Lc_i$ is called a {\em tautological quotient line bundle}
over $X$.

\begin{thm}[\cite{Kaveh}]
\label{T:iso}
The ring $R_P$ is isomorphic to the Chow ring (and to the cohomology ring)
of the complete flag variety $X$ for $GL_n(\C)$ (note that $\dim(X)=d$)
so that the images in $R_P$ of the differential operators
$\frac{\partial}{\partial\l_1},\ldots,\frac{\partial}{\partial\l_n}$
get mapped to the first Chern classes of the tautological quotient line bundles
$\Lc_1,\dots,\Lc_n$ over $X$.
\end{thm}

This theorem can also be deduced directly from the Borel presentation for
the cohomology ring $H^*(X,\Z)$ using that the volume of $P_\l$
(regarded as a function of $\l$) is equal to $\prod_{i<j}(\lambda_i-\lambda_j)$
times a constant.

Along with the Gelfand--Zetlin polytope $P$, we consider its resolution $Q$
such that the number of facets in $Q$ is the same as the number of facets in $P$,
and every support hyperplane of $Q$ intersecting $Q$ by a facet is sufficiently
close to a parallel support hyperplane of $P$ intersecting $P$ by a facet.
This establishes a one-to-one correspondence between facets of $Q$ and
facets of $P$ such that the corresponding facets are parallel.
Nothing in what follows will depend on a particular choice of $Q$.

\subsection{Faces and face diagrams}
It will be convenient to represent faces of $P$ by {\em face diagrams}.
First, replace all $\l_j$ and $\l_{i,j}$ in table $(GZ)$ by dots.
Every face of $P$ is given by a system of equations of the form $a=b$,
where $a$ and $b$ are coordinates represented by adjacent dots in two
consecutive rows.
To represent such an equation, we draw a line interval connecting the
corresponding dots (these line intervals go from northeast to southwest
or from northwest to southeast).
Thus a system of equations defining a face of $P$ gets represented by
a collection of line intervals called the face diagram.\footnote{
Our face diagrams (as well as the diagrams in \cite{VK})
are reflections of the diagrams in \cite{KThesis} in a horizontal line.}
{\em Rows} of a face diagram are defined as the collections of dots
corresponding to the coordinates $\l_{i,j}$ with a fixed $i$,
and {\em columns} are by definition collections of dots with a fixed $j$
(columns look like diagonals in our pictures).

Let $F$ be a regular face of $P$ and $\tilde F$ the corresponding face of $Q$,
so that $F$ is the degeneration of $\tilde F$.
We will often write $[F]$ for the class $[\tilde F]$ of $\tilde F$ in
the polytope ring $R_Q$. Note that, in general, $\pi[F]$ {\em does not belong to} $R_P$.

Every facet of $P$ is regular.
For $i=0,\dots,n-1$ and $j=1,\dots,n-i-1$, let
$\G_{i,j}$ denote the facet of $P$ given by the equation
$\l_{i,j}=\l_{i+1,j}$, where we set $\l_{0,j}=\l_j$.
Similarly, for $i=0,\dots,n-1$ and $j=2,\dots,n-i$,
we let $\G^-_{i,j}$ denote the facet $\l_{i,j}=\l_{i+1,j-1}$.
Clearly, any facet of $P$ is either one of $\G_{i,j}$ or one of
$\G^-_{i,j}$.

The following proposition describes all linear relations between
facets of $P$:

\begin{prop}
\label{P:linrels}
  We have the following linear relations in $R_Q$:
  $$
  [\G_{i,j}]-[\G^-_{i,j}]-[\G_{i-1,j}]+[\G^-_{i-1,j+1}]=0,
  $$
  where the terms must be ignored if their indices are out of range.
  Moreover, all linear relations are generated by these.
\end{prop}

We call the relation displayed above the {\em 4-term relation} at $(i,j)$.

\begin{proof}
  Let $e_{i,j}$ be the standard basis in $\R^d$.
  The 4-term relation at $(i,j)$ is exactly the relation of the form
  $$
  \sum_\Gamma \xi_\Gamma(e_{i,j})[\Gamma]=0,
  $$
  where the summation is over all facets of $Q$.
  In fact, there are at most four facets $\G$ of $Q$ such that
  $\xi_\G(e_{i,j})\ne 0$; these are $\G_{i,j}$, $\G_{i-1,j}$,
  $\G^-_{i,j}$ and $\G^-_{i-1,j+1}$.
  It is straightforward to check that the coefficients are as stated.
\end{proof}

\subsection{Kogan faces}
In what follows, we  will mostly consider faces of the
Gelfand--Zetlin polytope given by the equations of the type\footnote{i.e. of type $L$
in notation of \cite{VK}, which is the same as type $A$
equation in \cite{KThesis} (his $\l_{i+j,i}$ is our $\l_{i,j}$).}
$\l_{i,j}=\l_{i+1,j}$ (i.e. the intersections of facets of the form $\G_{i,j}$).
We will call such faces {\em Kogan faces}.
To each Kogan face $F$, we assign the permutation $w(F)$ as follows.
First, assign to each equation $\l_{i,j}=\l_{i+1,j}$ the simple reflection
$s_{i+j}=(i+j,i+j+1)$.
Now compose all simple reflections corresponding to the equations defining $F$ by
going from left to right in each row of the diagram for $F$ and by going from the
bottom row to the top one.
We say that a Kogan face $F$ is {\em reduced} if the decomposition for
$w(F)$ obtained this way is reduced.\footnote{
Note that our definition of $w(F)$ doesn't agree with \cite[2.2.1]{KThesis}:
his $w(F)$ is our $w(F)^{-1}$, but this difference  does not affect the definition
of reduced faces.}
Reduced Kogan faces of the Gelfand--Zetlin polytopes are in bijective
correspondence with reduced
{\em pipe-dreams} (see \cite[2.2.1]{KThesis} for more details).
Note that the permutations associated with
a face and with the corresponding pipe-dream are the same.

All reduced Kogan face diagrams for $n=3$ with the corresponding permutations
are shown in Figure 2.

\begin{figure}
  \includegraphics[width=3cm]{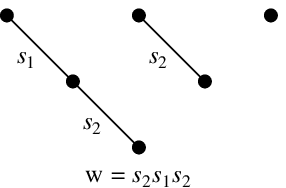}\qquad
  \includegraphics[width=3cm]{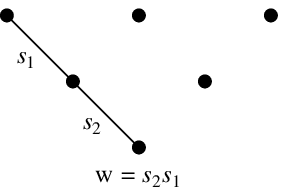}\qquad
  \includegraphics[width=3cm]{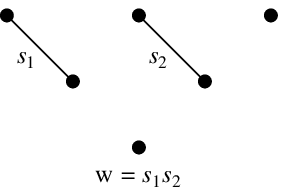}\\
  \bigskip
  \includegraphics[width=3cm]{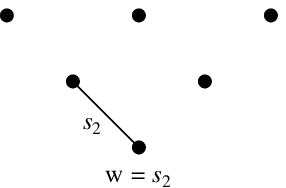}\qquad
  \includegraphics[width=3cm]{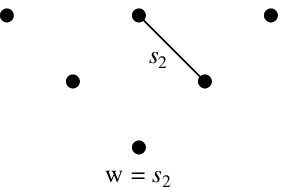}\qquad
  \includegraphics[width=3cm]{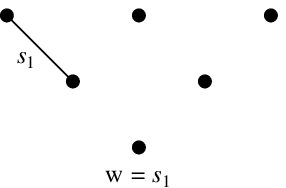}\qquad
  \includegraphics[width=3cm]{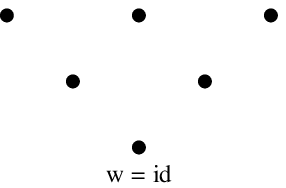}
  \caption{Reduced Kogan face diagrams for the 3-dimensional Gelfand--Zetlin polytope}
\end{figure}

\begin{prop}
  All Kogan faces are regular.
\end{prop}

\begin{proof}
  There is a unique Kogan vertex.
This vertex is simple and contained in any other Kogan face.
Now the result follows from Proposition \ref{P:reg}.
\end{proof}

Using the 4-term relations, we can express $[\G_{0,j+1}^-]$ through Kogan facets:
$$
[\G_{0,j+1}^-]=[\G_{0,j}]-[\G_{1,j}]+[\G_{1,j}^-]=
[\G_{0,j}]-[\G_{1,j}]+[\G_{1,j-1}]-[\G_{2,j-1}]+[\G_{2,j-1}^-]=\dots
$$
$$
\dots=\sum_{i=0}^{j-1}[\G_{i,j-i}]-[\G_{i+1,j-i}].
$$
Define the {\em $k$-antidiagonal sum of facets $AD_k$} as the sum of all elements of the form
$[\G_{i,j}]$ with $i+j=k$ being fixed (including the case $i=0$).
Set $\G_j=\G_{0,j}$ and $\G^-_j=\G^-_{0,j}$.
The computation we have just made shows that
$$
[\G_{j+1}]-[\G_{j+1}^-]=AD_{j+1}-AD_j.
$$

\begin{prop}
  \label{P:dla}
 We have the following identities in $R_P$:
 $$
 \frac{\d}{\d\l_1}=\pi(-[\G_1]),\ \
 \frac{\d}{\d\l_2}=\pi([\G^-_2]-[\G_2]),\ \dots\
 \frac{\d}{\d\l_n}=\pi([\G^-_n]).
 $$
\end{prop}

\begin{proof}
 Let $\d_j$ be the image of the vector $\frac{\d}{\d\l_j}$ under
 the natural inclusion $\Lambda_P\to\Lambda_Q$.
 Denote by $H_j$ and $H^-_j$ the support numbers corresponding
 to facets $\G_j$ and $\G^-_j$, respectively.
 Thus $H_j$ and $H_j^-$ are linear functionals on $\Lambda_Q$.
 By the chain rule, we have
 $$
 \d_j=\sum_{k=1}^{n-1} H_k(\d_j)[\G_j]+\sum_{k=2}^n H^-_k(\d_j)[\G^-_j],
 $$
 in $R_Q$ since $[\G_j]=\d/\d H_j$ and similarly for $[\G^-_j]$.
 It suffices to note that $H_k(\d_j)=-\delta_{kj}$ and
 $H^-_k(\d_j)=\delta_{kj}$, where $\delta_{kj}$ is the Kronecker delta.
\end{proof}

\section{Schubert cycles and faces}\label{s.cycles}
\subsection{Schubert cycles}\label{ss.Schubert}
For the rest of the paper, we set $G=GL_n(\C)$.
Let $B$ and $B^-$, respectively, be the subgroups of upper-triangular and
lower-triangular matrices in $G$.
The Weyl group of $G$ is identified with the symmetric group $S_n$:
a permutation $w\in S_n$ corresponds to the element of $G$ acting
on the standard basis vectors $e_i$ by the formula $e_i\mapsto e_{w(i)}$.
For each $w\in S_n$, we define the {\em Schubert variety} $X^w$ to be the closure of
the $B^-$--orbit of $w$ in the flag variety $X=G/B$. It is easy to check that the
length $l(w)$ of $w$ is equal to the codimension of $X^w$ in $X$. The class $[X^w]$
of $X^w$ in $CH^{l(w)}(X)$ is called the {\em Schubert cycle} corresponding to the
permutation $w$.
\begin{remark}
\label{r.notation1}
Note that notation in \cite{KM} is different from ours.
Namely, they consider the flag variety $B^{-}\setminus G$.
Under the isomorphism $G/B\to B^{-}\setminus G$ that sends $gB$ to
$w_0g^{-1}w_0^{-1}B^{-}$ our Schubert variety $X^w$ gets mapped to the
Schubert variety $X_{w_0^{-1}ww_0}$ in notation of \cite[\S4]{KM}.
Here $w_0\in S_n$ denotes the longest permutation $i\mapsto n-i+1$.
\end{remark}

\subsection{Schubert polynomials}
We now recall the notion of a {\em Schubert polynomial} \cite{BGG,LS}.
For every elementary transposition $s_i=(i,i+1)$, define the corresponding
divided difference operator (acting on polynomials in $x_1$, $x_2$, $\dots$)
by the formula
$$
A_i(f)=\frac{f-s_i(f)}{x_i-x_{i+1}},
$$
where $s_i(f)$ is the polynomial $f$ with variables $x_i$ and $x_{i+1}$ interchanged.
For a permutation $w$, consider a reduced (i.e. a shortest) decomposition
$w^{-1}w_0=s_{i_1}\dots s_{i_k}$ of $w^{-1}w_0$
into a product of elementary transpositions.
The Schubert polynomial $\Sf_w$ is defined by the formula
$$
\Sf_w(x_1,x_2,\dots)=A_{i_1}\dots A_{i_k}(x_1^{n-1}x_2^{n-2}\dots x_{n-1})
$$

\begin{thm}[\cite{BGG}]
\label{T:BGG}
 The class $[X^w]$ of the Schubert variety $X^w$ in $CH(X)$ is
 equal to $\Sf_w(x_1,x_2,\dots)$, where $x_i=-c_1(\Lc_i)$ is the
 negative first Chern class of the tautological quotient line bundle $\Lc_i$.
 Under our identification $CH(X)=R_P$, we have
 $$
 [X^w]=\Sf_w\left(-\frac{\d}{\d\l_1},\dots,-\frac{\d}{\d\l_n}\right).
 $$
\end{thm}

We now recall the Fomin--Kirillov theorem \cite{FomKir}.
Assign to each face $F$ the monomial
$x(F)$ in $x_1$,\ldots, $x_{n-1}$ by assigning $x_j$ to each equation
$\l_{i,j}=\l_{i+1,j}$ defining $F$ and then multiplying them all
(here the order, of course, does not matter).
The Fomin--Kirillov theorem states that the Schubert polynomial $\mathfrak S_w$ of the
Schubert cycle $[X^w]$ is equal to
$$
\sum_{w(F)=w}x(F),
$$
where the sum is only taken over reduced Kogan faces.

\subsection{Representation of Schubert cycles by faces}
The polytope ring provides a natural setting,
where Schubert cycles can be immediately identified
with linear combinations of faces sidestepping the use of Schubert polynomials.
The following theorem is a direct analog of the Fomin--Kirillov theorem:
it shows that each Schubert cycle can be represented
by the sum of faces in exactly the same way  as the respective
Schubert polynomial can be represented by the sum of monomials.

\begin{thm}
\label{t.FK}
The  Schubert cycle
$[X^w]$ regarded as an element of the Gelfand--Zetlin polytope ring can be represented
by the following linear combination of faces:
$$[X^w]=\pi\left(\sum_{w(F)=w}[F]\right),$$
where the sum is only taken  over reduced Kogan faces
(all these faces are regular).
\end{thm}

The proof of this theorem will be given in Subsection~\ref{ss.degrees}.
It uses combinatorics and geometry of the Gelfand--Zetlin polytope together with the Demazure character formula.

Despite the similarity between this theorem and the Fomin--Kirillov theorem,
the former can not be formally deduced from the latter.

Note that a Schubert cycle might have a simpler representation by sums of faces than
the one given by this theorem (see Example \ref{e.facets}).

\begin{example}
\label{e.facets}
Using Proposition \ref{P:dla} and Theorem \ref{T:BGG}, we can express the Schubert
divisors $[X^{s_i}]$ through the elements of the polytope ring corresponding to
facets of $P$.
First, we have
$$
[X^{s_{i_0}}]=\Sf_{s_{i_0}}\left(-\frac{\d}{\d\l_1},-\frac{\d}{\d\l_2},\dots\right)=
\pi\left(\sum_{j=1}^{i_0}[\G_{j}]-[\G^-_{j}]\right),
$$
where we drop all terms, whose indices are out of range.
As we have seen, the element $[\G_{j}]-[\G^-_{j}]$ equals to $AD_j-AD_{j-1}$.
It follows that
$$
[X^{s_{i_0}}]=\pi(AD_{i_0})=\pi\left(\sum_{j=1}^{i_0}[\G_{i_0-j,j}]\right).
$$
We obtain a representation of $[X^{s_{i_0}}]$ as a sum of $i_0$ facets.
This representation coincides with the one given in Theorem \ref{t.FK}.
Note, however, that $[X^{s_{n-1}}]$ can be represented by a single facet,
namely, we have
$$
[X^{s_{n-1}}]=-\sum_{i=1}^{n-1}\frac{\d}{\d\l_i}=\frac{\d}{\d\l_n}=\pi[\G^-_n].
$$
We have used the equality $\sum_{i=1}^{n}{\d}/{\d\l_i}=0$ in $R_Q$ because the
volume polynomial is translation invariant, in particular, it does not change
if we add the same number to all $\l_i$.
\end{example}

Theorem \ref{t.FK} together with relations in the polytope ring $R_P$
implies the following dual presentation of Schubert cycles by faces.
Define {\em dual Kogan faces} of the Gelfand--Zetlin polytope to be the faces given by
the equations of the type\footnote{
i.e. of type $R$
in notation of \cite{VK}, which is the same as type $B$ equation in \cite{KThesis}
}
$\l_{i,j}=\l_{i+1,j-1}$ (i.e. the intersections of facets of the form $\G_{i,j}^-$).
In other words, dual Kogan faces are mirror images of Kogan faces in a vertical line.
To each dual Kogan face $F^*$ we can again assign a permutation $w(F^*)$. Namely,
assign to each equation $\l_{i,j}=\l_{i+1,j-1}$ the simple reflection $s_{n-j+1}$ and
compose these reflections by going from the bottom row to the top one and from
right to left in each row. Note that the permutation $w(F^*)$ is the same as the
permutation $w(F)$ for the Kogan face $F$ obtained as the mirror image of $F^*$
in a vertical line (that is, each equation $\l_{i,j}=\l_{i+1,j-1}$ is replaced by $\l_{i,n-i-j+1}=\l_{i+1,n-i-j+1}$).

\begin{cor}
\label{c.FK}
The  Schubert cycle
$[X^w]$ regarded as an element of the Gelfand--Zetlin polytope ring can be represented
by the following linear combination of faces:
$$
[X^w]=\pi\left(\sum_{w(F^*)=w_0ww_0^{-1}}[F^*]\right),
$$
where the sum is only taken over reduced dual Kogan faces.
\end{cor}

\begin{proof}
Consider the linear automorphism of $\R^d$ that takes a point with coordinates
$\l_{i,j}$ to the point with coordinates $-\l_{i,n-i-j+1}$.
The automorphism takes a Gelfand--Zetlin polytope $P_\l$
to the Gelfand--Zetlin polytope $P_{-w_0\l}$, where $w_0(\l_1,\ldots,\l_n)=(\l_n,\ldots,\l_1)$.
Thus it induces an automorphism  of the space $V_P$ preserving the lattice $\Lambda_P$
and hence an automorphism $A$ of $R_P$.
Choose a resolution $Q$ so that the automorphism $A$ extends to $R_Q$.
It is clear that the extended automorphism takes the element $\pi[F]$ corresponding to
a regular face $F$ of $P$ to the element $\pi[F^*]$, where the face diagram of $F^*$
is obtained from the face diagram of $F$ by the mirror reflection in a vertical line.
It now suffices to prove that the automorphism $A$ of $R_P$ coincides with the automorphism
of $CH^*(X)$ that sends a Schubert cycle $[X^w]$ to $[X^{w_0ww_0^{-1}}]$. (The latter
automorphism is induced by the automorphism of $X$ that sends a complete flag to the
flag of orthogonal complements.)
Indeed, this is easy to verify for Schubert divisors as in Example \ref{e.facets}
(we basically need to repeat the same computation with dual Kogan faces
instead of Kogan faces).
The general case now follows since Schubert divisors are multiplicative generators of
the cohomology ring of $X$.
\end{proof}

Note that any Kogan face intersects any dual Kogan face transversally.
Hence, we can represent the cycles given by the Richardson varieties as sums of faces.

\begin{cor}
\label{c.Richardson}
The product of any two Schubert cycles $[X^w]$ and $[X^u]$ can be represented as the
sum of the following faces:
$$
[X^w]\cdot [X^u]=
\pi\left(\sum_{
               \genfrac{}{}{0pt}{2}{w(F)=w}{w(F^*)=w_0uw^{-1}_0}
               }
[F\cap F^*]\right),
$$
where $F$ and $F^*$ run over reduced Kogan and dual Kogan faces, respectively.
\end{cor}

\section{Demazure characters}\label{s.Demazure}

\subsection{Characters} For each $\l=(\l_1,\ldots,\l_n)$,
consider the affine hyperplane $\R^{n-1}\subset\R^n$ with coordinates $y_1$, \ldots,
$y_n$ given by the equation $y_1+\ldots+y_n+u_0=0$, where $u_0=\lambda_1+\dots+\lambda_n$.
Choose coordinates $u_1$, \ldots, $u_{n-1}$ in $\R^{n-1}$ such that $y_i=u_i-u_{i-1}$
for $i=1,\ldots, n-1$. Consider the following linear map $p:\R^d\to\R^{n-1}$ from the
space $\R^d$ with coordinates $\l_{i,j}$ to the hyperplane $\R^{n-1}\subset\R^n$:
$$
u_i=\sum_{j=1}^{n-i}\l_{i,j}.
$$
In other terms, if we arrange the coordinates $\l_{i,j}$ into
a triangular table as in $(GZ)$, then $u_i$ is the sum of
all elements in the $i$-th row.
In what follows, we identify $\R^n$ with the real span of
the weight lattice $\Lambda$ of $G$ so that the $i$-th basis vector in $\R^n$ corresponds to
the weight given by the $i$-th entry of the diagonal torus in $G$. Then the hyperplane
$\R^{n-1}$ is the parallel translate of the hyperplane spanned by the roots of $G$.
It is easy to check that the image of the Gelfand--Zetlin polytope $P_\l\subset \R^d$
under the map $p$ is the weight polytope of the representation $V_\l$.

Let $S$ be a subset of the Gelfand--Zetlin polytope $P_\l$ (in what follows $S$
will be a face or a union of faces).
Define the {\em character} $\chi_S$ of $S$ as the sum of formal exponentials
$e^{p(z)}$ over all integer points $z\in S$, that is,
$$
\chi(S):=\sum_{z\in S\cap\Z^d}e^{p(z)}.
$$
The formal exponentials $e^u$, $u\in\Z^n$, generate the group
algebra of  $\Lambda$.
Thus the character takes values in this group algebra.

Consider the linear operators $s_i:\R^{n-1}\to\R^{n-1}$
such that the point $s_i(u_1,\dots,u_{n-1})$ differs from
the point $(u_1,\dots,u_{n-1})$ at most in the $i$-th coordinate,
and the $i$-th coordinate of the point $s_i(u_1,\dots,u_{n-1})$
is equal to $u_{i-1}+u_{i+1}-u_i$, where $u_n=0$.
It is known and easy to verify that the operators $s_i$ are induced by the orthogonal
reflections of $\R^{n}$ (given by the simple roots) and that they generate an
action of the symmetric group $S_n$ on $\R^{n-1}$ such that the reflection $s_i$
corresponds to the elementary transposition $s_i=(i,i+1)$ (we use the same notation
for a reflection and a transposition, which is a standard practice when dealing with
group actions). We also define the action of $s_i$ on the group algebra of the weight
lattice by setting $s_i(e^u):=e^{s_i(u)}$.

In what follows, we identify $\R^{n-1}$ with the real vector
space spanned by the roots of $G$, and $s_i$ with the reflections corresponding
to simple roots. The simple roots correspond to the standard basis vectors in $\R^{n-1}$,
i.e. the only nonzero coordinate of the simple root $\a_i$ is $u_i$,
and this coordinate is equal to one.

Let $V^-_{\l, w}$ be the Demazure $B^-$--module defined as the dual space to the space of global sections $H^0(X^w,\Lc_\l|_{X^w})$,
where $\Lc_\l:=\Lc_1^{\otimes\l_1}\otimes\ldots\otimes\Lc_n^{\otimes\l_n}$ is the very ample line bundle on $X$ corresponding
to a strictly dominant weight $\l$.
Note that by the Borel--Weil--Bott theorem $V^-_{\l,e}$ is isomorphic to the irreducible representation $V_\l$ of $G$ with the highest weight $\l$.
Choose a basis of weight vectors in $V^-_{\l,w}$.
Recall that the {\em Demazure character} $\chi^w(\l)$
of $V^-_{\l,w}$ is the sum over all weight vectors in the basis of
the exponentials of the corresponding weights, or equivalently,
$$
\chi^w(\l):=\sum_{\mu\in\Lambda}m_{\l,w}(\mu)e^{\mu},
$$
where $m_{\l,w}(\mu)$ is the multiplicity of the weight $\mu$ in $V^-_{\l,w}$.

The main result of this section establishes a relation between the Demazure character of a Schubert variety
and the character of the union of the corresponding faces.

\begin{thm}
\label{t.Demazure}
For each permutation $w\in S_n$, the Demazure character
$\chi^w(\l)$ is equal to the character of the following union of faces:
$$\chi^w(\l)=\chi\left(\bigcup_{w(F_\l)=w}F_\l\right).$$
As usual, $F_\l$  runs only over reduced Kogan faces in the Gelfand--Zetlin polytope $P_\l$.
\end{thm}
Note that in contrast with Theorem \ref{t.FK}, this theorem and its corollaries below
(describing Hilbert functions and degrees of Schubert varieties in projective embeddings to $\P(V_\l)$)
use exactly the polytope $P_\l$ and not just any of the analogous to $P_\l$ polytopes.
Whenever the choice of $\l$ matters we indicate this by using notation $F_\l$ instead of $F$ for the faces.

For Kempf permutations, Theorem \ref{t.Demazure} reduces
to \cite[Corollary 15.2]{PS}.
Note that by \cite[Proposition 2.3.2]{KThesis} a permutation $w$ is Kempf
if and only if there is a unique reduced Kogan face $F$ such that $w(F)=w$.
Hence, $\chi^w(\l)=\chi(F)$ in this case.

In Subsection~\ref{ss.mitosis}, we will reduce this theorem to a purely combinatorial Lemma~\ref{T:comb}.
The proof of this lemma is given in Subsection~\ref{ss.keylemma}.

Let us now obtain several corollaries from Theorem \ref{t.Demazure}. First, we can similarly describe the Demazure character of $B$--modules.
Let $V^+_{\l, w}$ be the Demazure $B$--module defined as the dual space to $H^0(X_w,\Lc_\l|_{X_w})$, where
$X_w$ is the closure of the $B$--orbit of $w$ in $X$ (in particular, $[X_w]=[X^{w_0w}]$ in $CH^*(X)$).

\begin{cor}
\label{c.Demazure}
For each permutation $w\in S_n$, the Demazure character
$\chi_w(\l)$ is equal to the character of the following union of faces:
$$
\chi_w(\l)=\chi\left(\bigcup_{w(F_\l^*)=ww_0}F^*_\l\right),
$$
where $F_\l^*$ runs over reduced dual Kogan faces in the Gelfand--Zetlin polytope.
\end{cor}

The corollary follows immediately from the proof of Theorem \ref{t.Demazure} together
with the definition of dual Kogan faces since
$\chi_w(\l)=w_0\chi^{w_0w}(\l)$.

Another corollary from Theorem \ref{t.Demazure} describes the Hilbert function
of the Schubert variety $X^w$ embedded into $\P(H^0(X^w,\Lc_\l|_{X^w})^*)\subset\P(V_\l)$.

\begin{cor}
\label{c.Ehrhart}
For any permutation $w\in S_n$, the dimension of the space
$H^0(X^w,\Lc_\l|_{X^w})$
is equal to the number of integer points in the union of
all reduced Kogan faces with permutation $w$:
$$
\dim H^0(X^w,\Lc_\l|_{X^w})=\left|\bigcup_{w(F)=w}F_\l\cap\Z^d\right|.
$$
In particular, the Hilbert function $H_{\l,w}(k):=\dim H^0(X^w,\Lc_\l^{\otimes k}|_{X^w})$
is equal to the Ehrhart polynomial of $\bigcup_{w(F_\l)=w}F_\l$, that is,
$$
H_{\l,w}(k)=\left|\bigcup_{w(F_\l)=w}kF_\l\cap\Z^d\right|
$$
for all positive integers $k$.
\end{cor}

This corollary will be essential for the proof of Theorem~\ref{t.FK}.

\subsection{Degrees of Schubert varieties}\label{ss.degrees}
To prove Theorem \ref{t.FK} we first prove an analogous identity
for the {\em degree polynomials} of Schubert varieties.
The {\em degree polynomial} $D_w$ on $\R^n$ is uniquely characterized by the
property that $(d-l(w))!D_w(\l)=\deg_\l(X^w)$ for all dominant $\l\in\Z^n\subset\R^n$.
In particular, $D_{e}=\frac{1}{d!}\deg_\l(X)=\Volume(Q_\l)$ and $D_{w_0}=1$.
The degree polynomials originated in the work of Bernstein--Gelfand--Gelfand
\cite{BGG} and were recently studied by Postnikov and Stanley \cite{PS}.
Below we prove identities relating the degree polynomial and the volumes
of faces of the Gelfand--Zetlin polytope.

Denote by $\R F\subset\R^d$ the affine span of a face $F$.
In the formulas displayed below, the volume form on $\R F$ is normalized
so that the covolume of the lattice $\Z^d\cap \R F$ in $\R F$ is equal to 1. Then the following theorem holds:

\begin{thm}\label{t.deg}$$D_w=\sum_{w(F_\l)=w}\Volume(F_\l)$$
$$D_w=\sum_{w(F_\l^*)=w_0ww_0^{-1}}\Volume(F_\l^*)$$
\end{thm}

For Kempf permutations, the first equality of Theorem \ref{t.deg} reduces to the last formula from
\cite[Corollary 15.2]{PS}.

\begin{proof} Theorem \ref{t.deg} follows easily from Corollary \ref{c.Ehrhart} and Hilbert's theorem
describing the leading monomial of the Hilbert polynomial by the same arguments as in
\cite{Kh}.
Indeed, by Hilbert's theorem, $\dim(V^-_{k\l,w})$ is a polynomial in $k$ (for large $k$),
and its leading term is equal to $D_w(\l)k^d$.
Next, note that $\dim(V^-_{k\l,w})$ is the number of integer points in
$\bigcup_{w(F_\l)=w}kF_\l$  by Corollary \ref{c.Ehrhart}.
Finally, use that the volume of $F$ for each face $F$ is the
leading term in the Ehrhart polynomial of this face (since $\Volume(kF)=k^n \Volume(F)$ is
approximately equal to the number of integer points in $kF$ for large $k$).
\end{proof}

\begin{remark}
\label{r.KM}
Dual Kogan faces are exactly the faces considered in \cite[\S4]{KM}.
Note that the definition of $w(F^*)$ in \cite{KM} is different from ours
as well as from that of \cite{KThesis}.
Namely, in our notation they associate to a dual Kogan face
$F^*$ the permutation $w_0ww_0^{-1}$.
However, since their Schubert cycle
$[X_w]$ is defined so that it coincides with our Schubert cycle
$[X^{w_0ww_0^{-1}}]$ (see Remark \ref{r.notation1})
their Theorem \cite[Theorem 8]{KM} (describing the toric
degeneration of the Schubert variety $X_w$) uses exactly the same faces as
our the second equality of Theorem \ref{t.deg}, and the latter can be deduced from the former
by standard arguments from toric geometry.
\end{remark}

\begin{proof}[Proof of Theorem \ref{t.FK}] We now deduce Theorem \ref{t.FK} from Theorem \ref{t.deg} using Lemma \ref{L:Poincare}.
Recall that the lattice $\Lambda_{P}$ is a sublattice of $\Lambda_{Q}$.
In particular, the polytope $P_\l$ can be regarded as an element of
$\Lambda_{Q}=\Sym^1(\Lambda_{Q})$.
Let $L_\lambda$ denote the image of $P_\l$ under the canonical projection
$\Sym(\Lambda_{Q})\to R_{Q}$.
It is easy to check that, under the isomorphism of Theorem \ref{T:iso},
the class $\pi(L_\l)$ corresponds to the first Chern class of the line bundle
$\Lc_\l$.
Hence, we have the following identity in $R_P$:
$$
[X^w]\pi(L_\l)^{d-l} =(\deg_\l X^w)[pt],
$$
where $d-l=d-l(w)$ is the dimension of the variety $X^w$
(the product in the left-hand side is taken in $R_P$; according to
our usual convention, we identify elements of $R_P$ with their images
in $M_{Q,P}$).

On the other hand, it is easy to check that for any face $F_\l\subset Q_\l$ of
codimension $l$ we have that the product $[F_\l]L_\l^{d-l}$ in $R_Q$ is equal to
$(d-l)!\Volume(F_\l)$ times the class of a vertex.
Hence, by Theorem \ref{t.deg} we have
$$
[X^w]\pi(L_\l)^{d-l}=\pi(\sum_{w(F)=w}[F]L_\l^{d-l}).
$$
Since elements of the form $\pi(L_\l)^{d-l}$ span $R_P^{d-l}$,
we can apply Lemma \ref{L:Poincare} and conclude that
$[X^w]=\pi(\sum_{w(F)=w}[F])$.

Theorem~\ref{t.FK} is proved.
\end{proof}

\subsection{The Demazure character formula}
To prove Theorem \ref{t.Demazure}, we use the Demazure character formula
for $\chi^w(\l)$ together with a purely combinatorial argument.
We now recall the Demazure character formula (see \cite{A} for more details).
For each $i=1,\dots,n-1$, define the operator $T_i$ on
the group algebra of the weight lattice of $G$ by the formula
$$
T_i(f)=\frac{f-e^{-\a_i}s_i(f)}{1-e^{-\a_i}}.
$$
Similarly, define the operator $T^-_i$ by the formula
$$
T^-_i(f)=\frac{f-e^{\a_i}s_i(f)}{1-e^{\a_i}}.
$$

\begin{thm}[\cite{A}]
\label{t.Dem_formula}
Let $w=s_{i_1}\ldots s_{i_l}$ be a reduced decomposition of $w$.
Then the Demazure characters $\chi_w(\l)$ and $\chi^{w_0w}(\l)$
can be computed as follows:
$$
\chi_w(\l)=T_{i_1}\ldots T_{i_l}e^\l,
$$
$$
\chi^{w_0w}(\l)=T^-_{n-i_1}\ldots T^-_{n-i_l}e^{w_0\l}.
$$
\end{thm}

The first identity is the standard form of the Demazure character formula.
We will use the second identity, which follows immediately from the
first one since
$\chi_w(\l)=w_0\chi^{w_0w}(\l)$ and $w_0T_i=T_{n-i}^-w_0$.

Note that this theorem is similar to Theorem \ref{T:BGG} (and especially to its
K-theory version \cite{De}, see also \cite[\S3]{RP})
that describes Schubert cycles using the divided difference operators.
However, in the former theorem we apply
the operators $T_{i_j}$ in the same order as
elementary transpositions $s_{i_j}$ appear in a reduced decomposition of $w$,
while in the latter theorem the order is opposite (that is, the same as in $w^{-1}$).

\subsection{Mirror mitosis}\label{ss.mitosis}
{\em Mitosis} is a combinatorial operation introduced in \cite{KnM,Mi}
that produces a set of Kogan faces out of a Kogan face.\footnote{
The original definition was in terms of pipe-dreams rather than Kogan faces.}
If we apply mitosis (in the $i$-th column) to all reduced Kogan faces corresponding
to a permutation $w$, then we obtain all reduced Kogan faces corresponding
to the permutation $ws_i$, provided that $l(ws_i)=l(w)-1$.
We will need {\em mirror mitosis}, which is obtained from the mitosis
by the transposition of face diagrams (interchanging rows and columns). In other words,
mirror mitosis for $w$ is the usual mitosis for  $w^{-1}$. We use mirror mitosis
to deduce Theorem \ref{t.Demazure} from the Demazure character formula.
We now give a direct definition of the mirror mitosis.

Let $F$ be a reduced Kogan face of dimension $l$.
For each $i=1,\ldots,n-1$, we
construct a set $M^-_i(F)$ of reduced Kogan faces of dimension $l+1$ as follows. For each $i=1$,\ldots, $n-1$, we say that the diagram of $F$ has an edge in the $i$-th row  if the face $F$ satisfies an equation $\l_{i-1,j}=\l_{i,j}$ for some $j$. Similarly, we say that the diagram of $F$ has an edge in the $i$-th column
if there is an equation $\l_{j-1,i}=\l_{j,i}$ for some $j$.
Consider the $i$-th row in the face diagram of $F$.
If it does not have an edge in the first column, then $M^-_i(F)$ is empty.
Suppose now that the $i$-th row of $F$ contains edges in all columns
from the first to the $k$-th (inclusive), and does not have an edge in
the $(k+1)$-st column.
Then for each $j\le k$ we check whether the $(i+1)$-st row has an edge at
the $j$-th column.
If it does, we do nothing.
The elements of $M^-_i(F)$ correspond to the values of $j$, for
which there is no edge at the intersection of the $(i+1)$-st row
and the $j$-th column.
For such value of $j$, we delete the $j$-th edge in the $i$-th row and shift all
edges on the left of it in the same row one step south-east (to the $(i+1)$-st row)
whenever possible.
A new reduced Kogan face $F_{i,j}$ thus obtained is called the $j$-th
{\em offspring} of $F$ at the $i$-th row.
The set $M^-_i(F)$ consists of offsprings $F_{i,j}$ for all $1\le j\le k$.

The cardinality of $M^-_i(F)$ is equal to $k-k'$, where $k'$ is  the number of edges in the
first $k$ places of the $(i+1)$-st row. This is the same as the number of monomials
in $A_i(x_i^kx_{i+1}^{k'})$.
An illustration of mirror mitosis is given in Figure 3.

\begin{figure}\label{f.mitosis}
  \includegraphics[width=6cm]{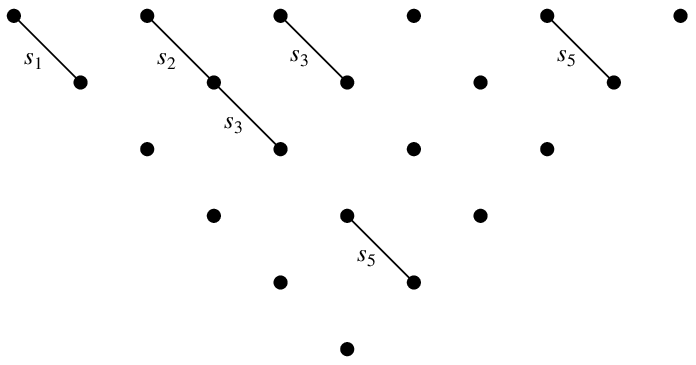}\\
\includegraphics[width=6cm]{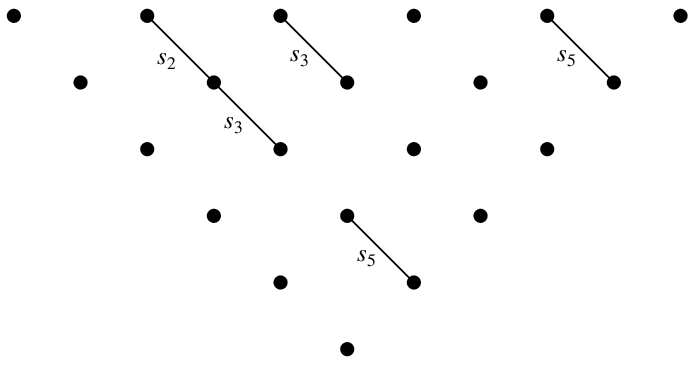}\qquad
\includegraphics[width=6cm]{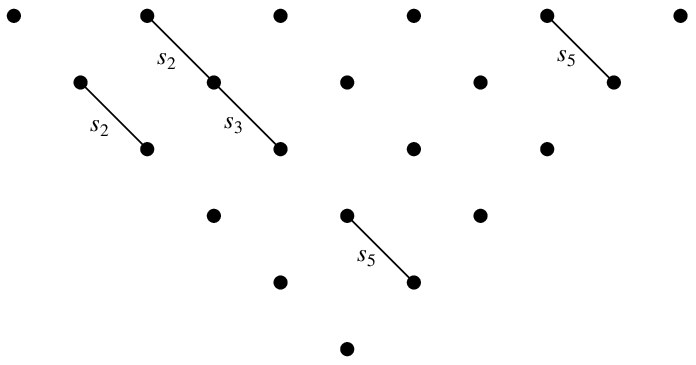}
\caption{The mirror mitosis applied to the first row of the face
shown above is the set of two faces shown below}
\end{figure}

The following theorem follows from the properties of the usual mitosis \cite{Mi}:

\begin{thm}
\label{t.mitosis}
If $l(s_iw)=l(w)-1$, then
$$
\bigcup_{w(F)=w}M^-_i(F)=\bigcup_{w(E)=s_iw}\{E\}.
$$
\end{thm}

\begin{varproof}[Proof of Theorem \ref{t.Demazure}] now follows by backwards induction on $l(w)$ from the
Demazure character formula
(the second identity of Theorem \ref{t.Dem_formula}), Theorem \ref{t.mitosis} and the following lemma.

\begin{keylemma}
\label{T:comb}
For each permutation $w\in S_n$ and an elementary transposition $s_i$
such that $l(s_iw)=l(w)-1$
we have
$$
T_i^-\chi(\bigcup_{w(F)=w}F_\l)=
\chi(\bigcup_{\genfrac{}{}{0pt}{2}{w(F)=w}{E\in M^-_i(F)}}E_\l).
$$
\end{keylemma}

The proof of this lemma is purely combinatorial.
It is given in Subsection \ref{ss.keylemma}.
\end{varproof}

\section{Mitosis on parallelepipeds}\label{s.keylemma}

In this section, we reduce the mitosis on the faces of the
Gelfand--Zetlin polytope to an analogous operation (called {\em paramitosis}) on the
faces of a parallelepiped.
The latter is easier to study and has a transparent geometric meaning
(see Remark \ref{r.simposis}).
Paramitosis for parallelepipeds and its applications to exponential sums and
Demazure operators are studied in Subsections \ref{ss:para} and \ref{ss:paramitosis}.
The material therein is self-contained, and all results are proved by elementary methods.
These results are then used in Subsection \ref{ss.keylemma} to prove Key Lemma
\ref{T:comb}.
Another application is Proposition \ref{p.cubic} that gives
a new minimal realization of a simplex as a cubic complex.

\subsection{Parallelepipeds}
\label{ss:para}
Consider integer numbers $\mu_1$, $\dots$, $\mu_m$, $\nu_1$, $\dots$, $\nu_m$, such that $\mu_k\le\nu_k$ for all $k=1,\dots,m$.
Define the parallelepiped $\Pi(\mu,\nu)$ as the convex polytope
$$
\Pi(\mu,\nu)=\{y=(y_1,\dots,y_m)\in\R^m\ |\ \mu_k\le y_k\le \nu_k,\ k=1,\dots,m\}.
$$
For any parallelepiped $\Pi=\Pi(\mu,\nu)$, consider the following sum
$$
S_\Pi(t)=\sum_{y\in\Pi\cap\Z^m}t^{\sigma(y)},\quad where\quad  \sigma(y)=\sum_{k=1}^m y_k.
$$
This is a polynomial in $t$.
It can be found explicitly:

\begin{prop}
\label{P:Sformula}
 We have
 $$
 S_\Pi(t)=\prod_{k=1}^m\frac{t^{\nu_k+1}-t^{\mu_k}}{t-1}.
 $$
\end{prop}

\begin{proof}
  Indeed,
  $$
  \sum_{y\in\Pi\cap\Z^m} t^{\sigma(y)}=\left(\sum_{y_1=\mu_1}^{\nu_1} t^{y_1}\right)
  \left(\sum_{y_2=\mu_2}^{\nu_2} t^{y_2}\right)\dots
  \left(\sum_{y_n=\mu_m}^{\nu_m} t^{y_m}\right).
  $$
  Each factor in the right hand side can be computed as a sum
  of a geometric series.
\end{proof}

The following is a duality property of $S_\Pi(t)$:

\begin{prop}
\label{P:dual}
  We have
  $$
  S_\Pi(t)=t^{\sum_{k=1}^m(\mu_k+\nu_k)}S_\Pi(t^{-1})
  $$
\end{prop}

The proof is a straightforward computation.
Proposition \ref{P:dual} can be restated in combinatorial terms as follows:
the number of ways to represent an integer $N$ as a sum $y_1+\dots+y_m$,
in which $\mu_k\le y_k\le \nu_k$ for all $k=1,\dots,m$, is equal to
the number of ways to represent the integer $\sum_{k=1}^m(\mu_k+\nu_k)-N$
in the same form.

Fix an integer $C$.
Consider the following linear operator on the space of Laurent
polynomials in $t$:
with every Laurent polynomial $f$, we associate the Laurent polynomial
$f^*$ obtained from $f$ by replacing every power $t^k$ with $t^{C-k}$.
In other terms, we have $f^*(t)=t^Cf(t^{-1})$.
Clearly, $f^{**}=f$ for every Laurent polynomial $f$.
The duality property of $S_\Pi$ can be restated as follows:
if $C=\sum_{k=1}^m(\mu_k+\nu_k)$, then $S_\Pi=S^*_\Pi$.
For the same value of $C$, define the operator $T$ by the formula
$$
T(f)=\frac{f-tf^*}{1-t}.
$$
It is not hard to see that, for every Laurent polynomial $f$, the function
$T(f)$ is also a Laurent polynomial.
The operator $T$ depends on the parallelepiped $\Pi$.

\begin{prop}
\label{P:SPi}
Let $\Gamma$ be the face of $\Pi=\Pi(\mu,\nu)$ given by the equation $y_1=\mu_1$
(it may coincide with $\Pi$ if $\mu_1=\nu_1$).
Then
$$
S_\Pi=T(S_\Gamma).
$$
\end{prop}

\begin{proof}
We have $\Gamma=\Pi(\mu_1,\mu_1,\mu_2,\nu_2,\dots,\mu_n,\nu_n)$.
Therefore, by Proposition \ref{P:Sformula},
$$
S_\Gamma(t)=t^{\mu_1}\prod_{k=2}^m\frac{t^{\nu_k+1}-t^{\mu_k}}{t-1}.
$$
and
$$
  S_\Pi(t)=\prod_{k=1}^m\frac{t^{\nu_k+1}-t^{\mu_k}}{t-1}=
  \frac{t^{\nu_1+1-\mu_1}S_\Gamma(t)-S_\Gamma(t)}{t-1}.
$$
Proposition \ref{P:dual} applied to $\Gamma$ gives that
$S_\Gamma(t)=t^{2\mu_1+\sum_{k=2}^m(\mu_k+\nu_k)}S_\Gamma(t^{-1})$.
Substituting this into the equation gives the desired result.
\end{proof}

Under certain assumptions, this proposition remains true if $\Pi$ and $\G$ are replaced
by their images under an embedding $\Pi\to\R^k$ that preserves the sum of coordinates.
\begin{prop}
  \label{P:Tgen}
  Consider a linear operator $\Lambda:\R^k\to\R^m$ defined over integers such that
  $\sigma\circ\Lambda=\sigma$ (the function $\sigma$ in the right hand side is the sum of
  all coordinate functions on $\R^k$).
  Let $\Pi$, $\Gamma$ and $T$ be as in Proposition \ref{P:SPi}.
  Assume that $\Lambda(B\cap\Z^k)=\Pi\cap\Z^m$ and $\Lambda(A\cap\Z^k)=\G\cap\Z^m$ for some
  subsets $A,B\subset\R^k$ such that the restrictions of
  $\Lambda$ to $B\cap\Z^k$ and $A\cap\Z^k$ are injective.
  Then
  $$
  \sum_{z\in B\cap\Z^k} t^{\sigma(z)}=
  T\left(\sum_{z\in A\cap\Z^k} t^{\sigma(z)}\right).
  $$
\end{prop}

\begin{proof}
  For every $z\in B\cap\Z^k$ set $y=\Lambda(z)$.
Since $\sigma(z)=\sigma(y)$, we obtain that
  $$
  \sum_{z\in B\cap\Z^k} t^{\sigma(z)}=\sum_{y\in\Pi\cap\Z^m} t^{\sigma(y)}
  $$
  (these two sums coincide term by term), and similarly for the right hand side.
  Thus the desired statement follows from Proposition \ref{P:SPi}.
\end{proof}

\subsection{Combinatorics of parallelepipeds}\label{ss:paramitosis}
Let $\Pi=\Pi(\mu,\nu)$ be a coordinate parallelepiped in $\R^m$ of dimension $m$,
so that $\mu_i<\nu_i$ for all $i=1,\dots,m$.
We will now discuss combinatorics of $\Pi$.
For every point $y\in\Pi$ with coordinates $(y_1,\dots,y_m)$, we can define the {\em paradiagram}
(``para'' from parallelepiped) of $x$ as the $m$-tuple $(\tilde y_1,\dots,\tilde y_m)$, where
\begin{itemize}
 \item $\tilde y_i=0$ if $y_i=\mu_i$,
\item $\tilde y_i=1$ if $y_i=\nu_i$, and
\item $\tilde y_i=*$ otherwise.
\end{itemize}
A paradiagram is called {\em reduced} if $1$ is never followed by $0$ in the paradiagram.

Consider a face $F$ of $\Pi$.
Note that all points in the relative interior of $F$ have the same paradiagram.
We will call this paradiagram the paradiagram of $F$.
A face $F$ is called reduced if so is its paradiagram.
Define a {\em parabox} as a sequence of consecutive positions in a paradiagram.
A parabox, filled with a string (possibly empty) of ones, followed by a single star,
followed by a string (possibly empty) of zeros, is called an {\em intron\footnote{The origin of this term is explained in \cite[Section 3.5]{KnM}} parabox}.
A parabox that contains the left end of a paradiagram and that is filled with
a string (possibly empty) of zeros is called an {\em initial parabox}.
A parabox that contains the right end of a paradiagram and that is filled with
a string (possibly empty) of ones is called a {\em final parabox}.
It is not hard to see that any reduced paradiagram consists of an initial parabox
followed by several (possibly zero) intron paraboxes, followed by a final parabox.
Below is an example of how to split a paradiagram into initial, intron and final paraboxes:

\medskip

\framebox[1.4cm]{0 0 0} \framebox[2.6cm]{1 1 1 $*$ 0 0} \framebox[1.4cm]{$*$ 0 0} \framebox[1.4cm]{1 1 $*$}
\framebox[0.6cm]{$*$} \framebox[1.4cm]{1 1 1}

\medskip

Two reduced faces $F_1$ and $F_2$ of $\Pi$ of the same dimension are said to be
related by an {\em $L$-move}, if their intersection is a non-reduced facet of
both $F_1$ and $F_2$.
We can also define an $L$-move of a reduced paradiagram.
This is the operation that replaces a single string $*0$ in a paradiagram with the string $1*$.
Note that an $L$-move does not affect the decomposition of a paradiagram into
initial, intron and final paraboxes.

\begin{prop}
 Two faces $F_1$ and $F_2$ of the same dimension are related by an $L$-move
if and only if their paradiagrams are related by an $L$-move.
\end{prop}

\begin{proof}
Let $\delta_1$ be the paradiagram of $F_1$, and $\delta_2$ the paradiagram of $F_2$.
Since $F_1\cap F_2$ has codimension 1 in $F_1$, the paradiagram $\delta$ of
$F_1\cap F_2$ is obtained from $\delta_1$ by replacing one star with either 0 or 1.
Consider two cases.

{\em Case 1:} the star is replaced with $0$.
Then, since $F_1\cap F_2$ is non-reduced, there must be a $1$ immediately before this $0$.
Since $F_2$ is reduced, this $1$ must be replaced with a star in the paradiagram $\delta_2$.
Therefore, the paradiagram $\delta_1$ is obtained from $\delta_2$ by an $L$-move.

{\em Case 2:} the star is replaced with $1$.
Then, since $F_1\cap F_2$ is non-reduced, there must be a $0$ immediately after this $1$.
Since $F_2$ is reduced, this $0$ must be replaced with a star in the paradiagram $\delta_2$.
Therefore, the paradiagram $\delta_2$ is obtained from $\delta_1$ by an $L$-move.
\end{proof}

Two faces of the same dimension are said to be {\em $L$-equivalent} if one
of them can be obtained from the other by a sequence of $L$-moves or
inverse $L$-moves (on the level of paradiagrams, inverse $L$-moves are
defined as the inverse operations to $L$-moves).
For the sake of brevity, we will write $L$-classes instead of $L$-equivalence classes.
Throughout the rest of the subsection, we identify $L$-classes of faces and their unions
(clearly, an equivalence class can be easily recovered from its union).

\begin{prop}\label{p.cubic}
The $L$-classes form a simplicial cell complex combinatorially
equivalent to the standard simplex.
More precisely:
\begin{itemize}
\item any $L$-class is homeomorphic to a closed disk,
\item there is a one-to-one correspondence between $L$-classes
and the faces of a simplex such that corresponding sets are homeomorphic, and
intersections correspond to intersections.
\end{itemize}
\end{prop}
Figure 4 illustrates the proposition for $m=3$.
\begin{figure}
\includegraphics[width=6cm]{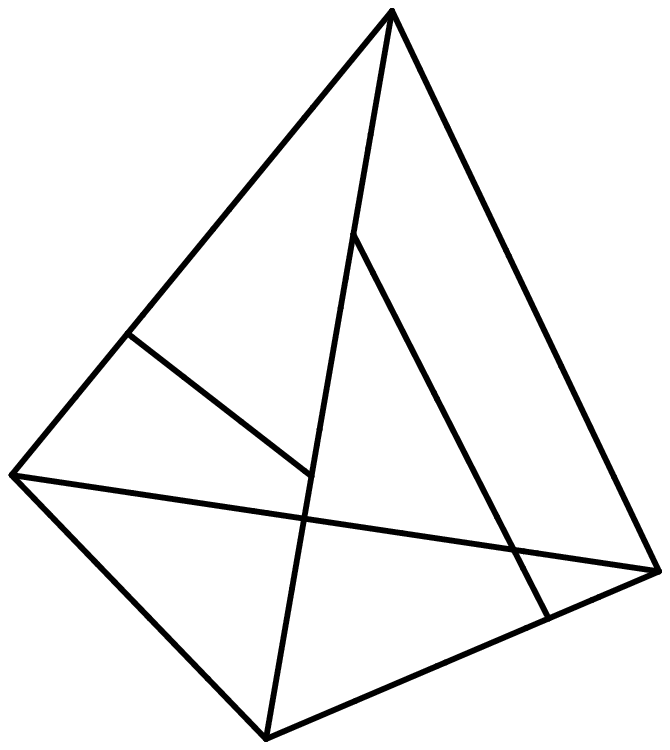}\qquad
\includegraphics[width=6cm]{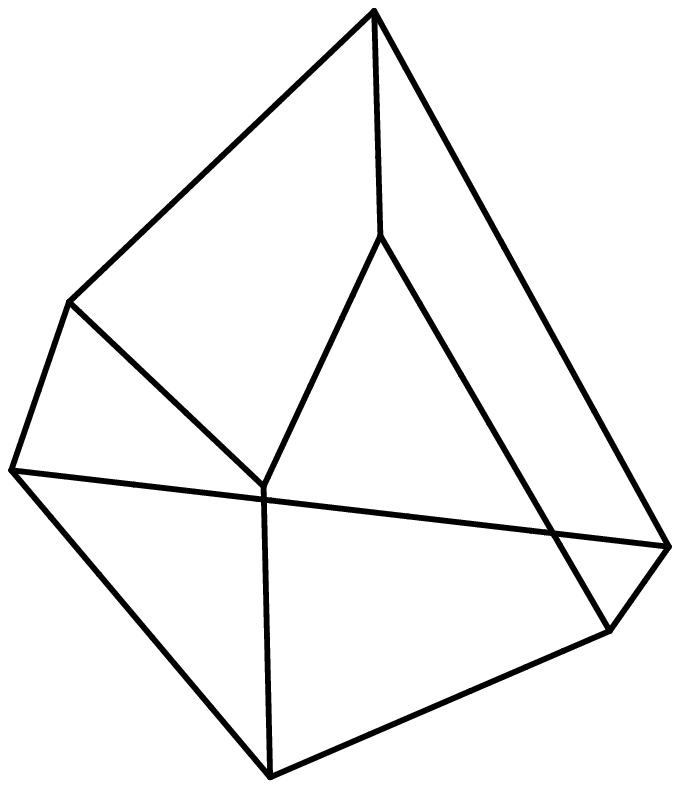}
\caption{The subdivision of the tetrahedron by two extra edges yields a
combinatorial cube}
\end{figure}

\begin{proof}
 First consider all reduced vertices.
There are exactly $m+1$ of them.
The paradiagram of a reduced vertex consists of a string of zeros followed by a string of ones.
Note that different reduced vertices are never $L$-equivalent.

Next, consider any $L$-class $A$ of dimension $k$. It has $k$ intron paraboxes.
To the $L$-class $A$, we assign a set $v(A)$ of $k+1$ vertices in the
following way: fill first $i\le k$ intron paraboxes with zeros, and the remaining
intron paraboxes with ones.
Clearly, the set $v(A)$ is precisely the set of all reduced vertices contained in the class $A$.
It follows that $v(A\cap B)=v(A)\cap v(B)$ for any two classes $A$ and $B$.
Note that $A$ is determined by the positions and sizes
of initial, intron and final paraboxes, that is, by the set $v(A)$.
The set $v(A)$ spans a face of the simplex $\Sigma$ with vertices at all reduced vertices of $\Pi$.
Thus we have an injective map $v$ from $L$-classes of faces of $\Pi$ to faces of $\Sigma$;
this map takes intersections to intersections.

The map $v$ is surjective: any set of reduced vertices has the form $v(A)$ for some equivalence class $A$.
Indeed, $A$ can be defined as the class, in which the boundaries of intron paraboxes are
the boundaries between zeros and ones in the vertices from the given set.

It remains to prove that any $L$-class is homeomorphic to a closed disk.
First note that an $L$-class with only one intron parabox is a broken line,
whose straight line segments are parallel to coordinate axes (every straight line
segment of this broken line corresponds to a particular position of the star
inside the intron parabox).
A broken line is homeomorphic to the interval.
In general, an $L$-class is a direct product of broken lines as above,
hence it is homeomorphic to a direct product of intervals, i.e. to a closed cube.
\end{proof}

The most important corollary for us is that the intersection of two
$L$-classes is again an $L$-class.

We can now define {\em paramitosis}.
This is an operation that produces several faces out of a single face $F$.
If the paradiagram of $F$ has no initial parabox, then the paramitosis of $F$ is empty.
Suppose now that the paradiagram of $F$ has a nonempty initial parabox.
Then we replace it with an intron parabox: the set of all faces obtained in
this way (corresponding to all different ways of filling the new intron parabox)
is the paramitosis of $F$. Below is an example of paramitosis:
\medskip

\framebox[1.4cm]{0 0 $*$ 0}\quad
$
\xymatrixcolsep{5pc}
\xymatrix{ \ar[r]^{\hbox{paramitosis}} &  }$
\quad
\framebox[1.4cm]{$*$ 0 $*$ 0}\quad \hbox{and}\quad \framebox[1.4cm]{1 $*$ $*$ 0}

\medskip
The paramitosis of a set of faces is defined as the union of paramitoses
of all faces in this set.

\begin{remark} \label{r.simposis}
It is easy to describe the paramitosis of an $L$-class using the bijection between
$L$-classes and faces of the standard simplex defined in Proposition \ref{p.cubic}.
Namely, the $L$-classes with non-empty initial paraboxes correspond to the faces
of the simplex contained in a facet $H$.
Let $v$ be the vertex of the simplex that is not contained in $H$.
Then the paramitosis of the face $A\subset H$  coincides with the convex hull of $A$ and $v$.
It follows that paramitosis of an $L$-class is again an $L$-class and that paramitosis
of the intersection of two $L$-classes with non-empty initial paraboxes coincides with
the intersection of their paramitoses.

\end{remark}

For a subset $A\subset\Pi$, we define the Laurent polynomial $\Sc(A)=\sum_{y\in A\cap\Z^m} t^{\sigma(y)}$.

\begin{prop}
\label{P:paramitosis}
Let $T$ be the operator associated with $\Pi$ as in Subsection \ref{ss:para},
the function $\sigma:\R^m\to\R$ the sum of all coordinates, and $A$ an $L$-class of
faces in $\Pi$ with a non-empty initial parabox.
Let $B$ be the paramitosis of $A$.
Then $\Sc(B)=T\Sc(A)$.
\end{prop}

\begin{proof}
Consider the paradiagram of a face from $A$.
Suppose that this paradiagram has $r$ paraboxes total, and the $l$-th parabox starts with
index $j_l$ (so that $j_1=1$).
Consider the following linear map $\Lambda_F:\R^m\to\R^r$:
$$
\Lambda_F(y_1,\dots,y_m)=\left(\sum_{j=j_1}^{j_2-1}y_j,\sum_{j=j_2}^{j_3-1}y_j,\dots,\sum_{j=j_r}^{m}y_j\right)
$$
We have $\sigma\circ \Lambda_F=\sigma$, where $\sigma$ is the function computing the sum of
all coordinates.
We can now apply Proposition \ref{P:Tgen} to the map $\Lambda_F$.
\end{proof}

A similar statement holds for unions of $L$-classes:
\begin{prop}\label{P:nonempty}
 Let $A_1$, $\dots$, $A_k$ be $L$-classes with nonempty initial paraboxes,
and suppose that the $L$-classes $B_1$, $\dots$, $B_k$ are obtained from  $A_1$, $\dots$, $A_k$
by paramitosis. Then we have
$$
\Sc(B_1\cup\dots\cup B_k)=T\Sc(A_1\cup\dots\cup A_k)=T\Sc(B_1\cup\dots\cup B_k).
$$
\end{prop}

\begin{proof}
We will use the inclusion-exclusion formula:
$$
\Sc(A_1\cup\dots\cup A_k)=\sum_{I\ne\emptyset}(-1)^{|I|-1}\Sc(A_I),
$$
where the summation is over all nonempty subsets $I\subset\{1,\dots,k\}$, and
$A_I$ is the intersection of all $A_i$, $i\in I$.
The same formula holds for $B_I$, and $T$ is linear, hence it suffices to show that
$\Sc(B_I)=T\Sc(A_I)$.
But $A_I$ is also an $L$-class with nonempty initial parabox.
Thus the first equality follows from Proposition \ref{P:paramitosis}.

The second equality follows from the first one since $T\circ T=T$.
\end{proof}

Let $M(A)$ denote the paramitosis of $A$.
\begin{prop}
\label{P:paramitosis1}
Suppose that $A_1$, $\dots$, $A_k$ are $L$-classes with nonempty initial parabox,
and $B_1$, $\dots$, $B_{r}$ are $L$-classes with empty initial parabox.
Suppose also that for every $i\in\{1,\dots,r\}$, we have
$B_{i}=M((A_1\cup\dots\cup A_k)\cap B_{i})$.
Then
$$
\Sc M(A_1\cup\dots\cup A_k\cup B_1\cup\dots\cup B_{r})=T\Sc(A_1\cup\dots\cup A_k\cup B_1\cup\dots\cup B_{r}).
$$
\end{prop}

\begin{proof}
By the inclusion-exclusion formula, we have for the right-hand side ($RHS)$:
$$
RHS=T\Sc(A_1\cup\dots\cup A_k\cup B_1\cup\dots\cup B_{r})=T\Sc(A_1\cup\dots\cup A_k)+$$
$$+T\Sc(B_1\cup\dots\cup B_{r})-T\Sc((A_1\cup\dots\cup A_k)\cap(B_1\cup\dots\cup B_{r}))
$$
Put $A_i'=(A_1\cup\dots\cup A_k)\cap B_{i}$ for every $i\in\{1,\dots,r\}$. Since $B_{i}=M(A_i')$, we have that $T\Sc(B_1\cup\dots\cup B_{r})=T\Sc(A'_1\cup\dots\cup A'_r)$ by the second equality of Proposition \ref{P:nonempty}. Hence,
$$T\Sc(B_1\cup\dots\cup B_{r})=T\Sc((A_1\cup\dots\cup A_k)\cap(B_1\cup\dots\cup B_{r})),$$
and $RHS=T\Sc(A_1\cup\dots\cup A_k)$.

It remains to note that the left-hand side coincides with $\Sc M(A_1\cup\dots\cup A_k)$ because $M(B_1\cup\dots\cup B_{r})$ is empty. The desired statement now follows from the first equality of Proposition \ref{P:nonempty}.
\end{proof}
\begin{remark}\label{r:paramitosis}
Note that the condition $B=M((A_1\cup\dots\cup A_k)\cap B)$ in Lemma \ref{P:paramitosis1} is satisfied
whenever $B=M(A)$ for an $L$-class $A\subset A_1\cup\dots\cup A_k$.
Indeed, if $B=M(A)$ then  by definition of paramitosis $A=H\cap B$, where $H$ is the hyperplane $y_1=\mu_1$.
Since $H$ contains all $A_i$, we always have the inclusion $(A_1\cup\dots\cup A_k)\cap B \subset H\cap B$.
On the other hand, the condition $A\subset A_1\cup\dots\cup A_k$ implies the opposite  inclusion
$H\cap B\subset (A_1\cup\dots\cup A_k)\cap B$.
\end{remark}

\subsection{Fiber diagrams, ladder moves and proof of Key Lemma \ref{T:comb}} \label{ss.keylemma}
We now apply the general results for parallelepipeds to
mitosis on faces of the Gelfand--Zetlin polytopes $P_\lambda$.
Fix some $i$.
We will now consider mirror mitosis in the $i$-th row (in what follows, by
mitosis, we always mean mirror mitosis).
Define the linear projection $q_i:\R^d\to \R^{d-(n-i)}$ forgetting
all entries in the $i$-th row, i.e. forgetting the values of
all coordinates $\lambda_{i,j}$ with first index $i$.
Define the {\em fibers of $P_\l$} as the fibers of this projection
restricted to the Gelfand--Zetlin polytope $P_\l$.

Fix the values of all coordinates $\l_{i',j}$ with $i'\ne i$.
This defines a fiber of $P_\l$.
The fiber can be given in coordinates $y_j=\l_{i,j}$ by
the following inequalities:
$$
\begin{array}{ccccccccc}
\l_{i-1,1}&   &\l_{i-1,2}&   &\l_{i-1,3}&       &\ldots&   &\l_{i-1,n-i+1}\\
     &y_1&     &y_2&     &\ldots &      &y_{n-i}&         \\
     &   &\l_{i+1,1}&   &\ldots&      &\l_{i+1,n-i-1}&  &       \\
\end{array}.
$$
Set $\mu'_j=\max(\l_{i-1,j},\l_{i+1,j-1})$ and
$\nu'_j=\min(\l_{i-1,j+1},\l_{i+1,j})$, where
$\l_{i+1,0}=-\infty$ (or sufficiently large negative number)
and $\l_{i+1,n-i}=+\infty$ (or sufficiently large positive number).
Therefore, the fiber can be identified with the coordinate parallelepiped
$\Pi(\mu',\nu')\subset\R^{n-i}$.

Let $F$ be any reduced Kogan face of $P_\l$.
We define the fiber of the face $F$ as the intersection of $F$ with the fiber of $P_\l$.
It will be convenient to represent the fiber of $F$ by the $i$-th {\em fiber diagram}
of $F$, i.e. by the restriction of the face diagram of $F$ to the union of rows
$i-1$, $i$, $i+1$.
Note that the mitosis in the $i$-th row can be seen on the level of
the fiber diagram --- it does not change other parts of the face diagram.
With the fiber diagram of every Kogan face, we can associate a paradiagram of a face of $\Pi(\mu',\nu')$ as follows.
The fiber of every Kogan face is a face of the parallelepiped $\Pi(\mu',\nu')$,
and we take the paradiagram of this face (note that the length of such paradiagram, which is equal
to the dimension of $\Pi(\mu',\nu')$, may be strictly less than $n-i$).
It is easy to check that the paradiagram of a reduced Kogan face is also reduced,
and that the mirror mitosis on the level of fiber diagrams coincides with the
paramitosis on the associated paradiagrams.

For the convenience of the reader , we will now recall the definition of {\em ladder-move} of \cite{BB} in
the language of reduced Kogan faces.
Consider rows $i-1$, $i$ and $i+1$ in the face diagram of $F$.
Define a {\em diagonal} as a collection of 3 dots in rows $i-1$, $i$ and $i+1$
that are aligned in the direction from northwest to southeast, together
with all connecting segments between these 3 dots that belong to the face diagram.
Diagonals can be of four possible types: $(0,0)$, $(0,1)$, $(1,0)$ and $(1,1)$.
The first entry is one if the diagonal contains an interval connecting rows $i-1$ and $i$,
otherwise the first entry is zero.
The second entry is one if the diagonal contains an interval connecting rows $i$ and $i+1$,
otherwise the second entry is zero:

\medskip

\centerline{
\includegraphics[width=1.5cm]{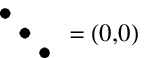}\qquad
\includegraphics[width=1.5cm]{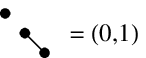}\qquad
\includegraphics[width=1.5cm]{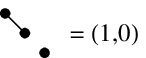}\qquad
\includegraphics[width=1.5cm]{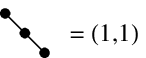}
}

\medskip

The correspondence between fiber diagrams and paradiagrams can now be described in combinatorial terms as follows:
diagonals of type $(0,0)$,  $(0,1)$, $(1,0)$ are replaced by $*$, $1$, $0$, respectively, and  diagonals of type
$(1,1)$ are ignored (each such diagonal decreases by one the dimension of the parallelepiped $\Pi(\mu',\nu')$,
that is, the length of the paradiagram).
For instance, the 1st fiber diagram of the upper face on Figure 3 yields the paradiagram  \framebox[1.4cm]{0 0 $*$ 0}.

Define a {\em box} as any sequence of consecutive diagonals in a fiber diagram.
In our pictures, a box will look like a parallelogram with angles $45^\circ$ and $135^\circ$.
By definition, a {\em ladder-movable box} is a box, whose first (left-most) diagonal
is of type (0,0), which follows by any number of type $(1,1)$ diagonals, and, finally,
by a single diagonal of type $(1,0)$.
Symbolically, we represent such a box as a sum $(0,0)+k(1,1)+(1,0)$, where $k$
is the number of type $(1,1)$ diagonals.
The {\em ladder-move} of \cite{BB} makes this ladder-movable box into the box $(0,1)+k(1,1)+(0,0)$:

\medskip

\centerline{
\includegraphics[width=10cm]{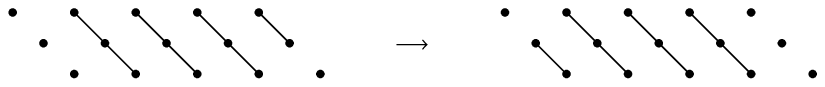}
}

\medskip

Note that ladder moves do not change the permutation associated with a face.
Moreover, they take reduced faces to reduced faces. Finally, note that under the correspondence
between fiber diagrams and paradiagrams, the ladder-moves are exactly the $L$-moves of the previous subsection.

We are now ready to prove Key Lemma \ref{T:comb}.
Denote by $\G''$ the set $\bigcup_{w(F)=w}F$ and by $\Pi''$ the union of all faces that
can be obtained from faces in $\G''$ by the mirror mitosis in the $i$-th row.
These are sets considered in Lemma \ref{T:comb}, and to prove the lemma we have to show that
$$
T_i^-(\chi(\G''))=\chi(\Pi'').
 $$

Let $\G'$ and $\Pi'$ denote fibers of $\G''$ and $\Pi''$, respectively,
at the $i$-th row.
Then Lemma \ref{T:comb} can be deduced from

\begin{lemma}
\label{L:fiber1}
Let $T$ be the operator associated with the coordinate parallelepiped $\Pi(\mu',\nu')$ as
in Subsection \ref{ss:para}. Identifying $\G'$ and $\Pi'$ with the subsets of
$\Pi(\mu',\nu')$, we get
$$
\sum_{y\in \Pi'\cap\Z^{n-i}} t^{\sigma(y)}=
  T\left(\sum_{y\in \G'\cap\Z^{n-i}} t^{\sigma(y)}\right).
$$
\end{lemma}

\begin{proof}[Proof of Key Lemma \ref{T:comb} using Lemma \ref{L:fiber1}]
Note that $q_i(\Pi')$ is a single point $z\in\R^{d-(n-i)}$
(i.e. all coordinates in all rows except for row $i$ are fixed).
For a point $x\in \Pi'$, denote by $y=(y_1,\ldots,y_{n-i})$ the coordinates of $x$ in
row $i$.
Let $\sigma_j(z)=\sigma_j(x)$ (for $j\ne i$) be the sum of coordinates in row $j$.
By definitions of $T_i^-$ and $T$, the following identity holds for all
$x\in\Pi'$ after substituting $t=e^{\a_i}$:
$$T_i^-e^{p(x)}=\prod_{j\ne i}e^{\sigma_j(z)\a_j}T(t^{\sigma(y)}).$$
In Lemma \ref{L:fiber1}, replace $t$ with $e^{\a_i}$,
and multiply both parts by the product
$$
\prod_{j\ne i} e^{\sigma_j(z)\a_j}.
$$
To obtain Key Lemma~\ref{T:comb}, it now suffices to take the sum over
all fibers $\Pi'$ of $\Pi''$ at the $i$-th row.
\end{proof}

\begin{proof}[Proof of Lemma \ref{L:fiber1}]
Lemma \ref{L:fiber1} will follow from Proposition \ref{P:paramitosis1} once
we check that $\G'$ satisfies the hypothesis of the proposition.
We know that $\G'$ is closed under $L$-moves since $\G''$ is closed under ladder-moves.
We can split $\G'$ into the union of $L$-classes
$A_1\cup\ldots\cup A_k\cup B_1\cup\ldots\cup B_r$ where $A_i$ and $B_i$, respectively, have nonempty and empty
initial parabox. By Remark \ref{r:paramitosis} it suffices to show for each $i\in\{1,\ldots,r\}$
that $B_i=M(A'_i)$ for some $A'_i\subset(A_1\cup\ldots\cup A_r)$.
This follows from the lemma:

\begin{lemma}\label{l.global}
Let $F$ be a reduced Kogan face such that $w(F)=w$ and the $i$-th paradiagram of $F$
has empty initial parabox.
If $l(s_iw)=l(w)-1$, then there exists another reduced Kogan face $F'$ such that
$w(F')=w$, the $i$-th paradiagram of $F'$ has nonempty initial parabox,
and $F\cap\Pi(\mu',\nu')$ is contained in $M(F'\cap\Pi(\mu',\nu'))$.
\end{lemma}

\begin{proof} Recall that the face diagram of $F$ defines a reduced decomposition
$w=s_{i_1}\cdots s_{i_l}$,
which by definition splits into two reduced words $w_1$ and $w_2$ as follows.
The word
$w_1=s_{i_1}\cdots s_{i_p}$
is composed from elementary transpositions by going from the bottom row to the $i$-th row
inclusively, and
$w_2=s_{i_{p+1}}\cdots s_{i_l}$ is composed by going from the $(i-1)$-st row to the top row.
In particular, the word $w_1$ contains only $s_j$ with $j\ge i$.
Since the initial parabox of $F$ is empty, the fiber diagram of $F$ starts with a
sequence of length $q$ of type $(1,1)$ diagonals followed by a type $(0,0)$ or $(0,1)$
diagonal.

If $q=0$, then the word $w_1$ contains only $s_j$ with $j>i$, in particular,
$w_1(i)=i$ and $(i+1,w_1^{-1}(i+1))$ is an inversion for $w_1$ unless $w_1(i+1)=i+1$.
Hence, the hypothesis  $l(s_iw)<l(w)$ (which is equivalent to $w^{-1}(i)>w^{-1}(i+1)$)
implies that $l(s_iw_2)<l(w_2)$. Indeed,
$$w^{-1}(i)=w_2^{-1}(i)>w_2^{-1}w_1^{-1}(i+1)\ge w_2^{-1}(i+1)$$
(the last inequality holds because $w=w_1w_2$ is reduced).
If $q>0$, the word $w_1$ can be further split as
$w_1's_{i+1}s_{i+2}\cdots s_{i+q}s_{i}s_{i+1}\cdots s_{i+q-1}w_1''$,
where $w_1'$ contains only $s_j$ with $j>i$ and $w_1''$ contains only $s_j$ with $j>i+q$.
By similar arguments we deduce that $l(s_{i+q}w_2)<l(w_2)$
(we use the identity
$s_i(s_{i+1}s_{i+2}\cdots s_{i+q}s_{i}s_{i+1}\cdots s_{i+q-1})=
(s_{i+1}s_{i+2}\cdots s_{i+q}s_{i}s_{i+1}\cdots s_{i+q-1})s_{i+q}$).

By applying to
$w_2=s_{i_{p+1}}\cdots s_{i_l}$
the exchange property we can replace it with a reduced word
$w_2'=s_{i+q}s_{i_{p+1}}\cdots\hat s_{i_r} \cdots s_{i_l}$.
We now replace $F$ with a reduced face $F'$ with the same permutation and non-empty
initial parabox as follows. In the face diagram of $F$, delete the edge corresponding
to $s_{i_r}$ and add the new edge $\l_{i,q+1}=\l_{i-1,q+1}$.
The resulting face diagram defines the face $F'$.
By construction, $M(F'\cap\Pi(\mu',\nu'))$ contains $F\cap\Pi(\mu',\nu')$.
\end{proof}
We now return to the proof of Lemma \ref{L:fiber1}.
Apply Lemma \ref{l.global} to a reduced Kogan face $F\in\G''$ whose paradiagram $B$ lies in $B_i$.
We get a face $F'\in\G''$ such that its paradiagram $A$ lies in $A_1\cup\ldots\cup A_r$ and $M(A)=B$.
Hence, the $L$-equivalence class $A'_i$ of $A$ also lies in $A_1\cup\ldots\cup A_r$ and
$M(A'_i)=B_i$ as desired.
\end{proof}

\end{document}